\documentclass[sn-nature,pdflatex]{sn-jnl}

\usepackage{graphicx}%
\usepackage{booktabs}
\usepackage{multirow}%
\usepackage{amsthm,amssymb,amsmath,amsfonts}
\usepackage[scr=boondox]{mathalfa}
\usepackage{mathtools}
\usepackage{dsfont}

\usepackage{arydshln}
\usepackage[title]{appendix}%
\usepackage{xcolor}
\usepackage{lmodern}
\usepackage{enumitem}
\usepackage{tikz}

\let\emptyset\varnothing

\newtheorem{thm}{Theorem}

\newtheorem{ass}[thm]{Assumption}
\newtheorem{defi}[thm]{Definition}
\newtheorem{col}[thm]{Corollary}
\newtheorem{prop}[thm]{Proposition}
\newtheorem{rem}[thm]{Remark}
\newtheorem{lemm}[thm]{Lemma}

\numberwithin{thm}{section} 

\usepackage[acronym,nowarn,hyperfirst=false]{glossaries}
\setkeys{glslink}{hyper=false}
\newacronym{nlp}{NLP}{nonlinear programming}
\newacronym{itd}{ITD}{integrated transmission-distribution}
\newacronym{qp}{\textsc{qp}}{quadratic program}
\newacronym{ocd}{\textsc{ocd}}{Optimality Condition Decomposition}
\newacronym{app}{\textsc{app}}{Auxiliary Problem Principle}
\newacronym{sqp}{\textsc{sqp}}{Sequential Quadratic Programming}
\newacronym{sdp}{SDP}{semi-definite programming}
\newacronym{soc}{SOC}{second-order cone}
\newacronym{qcqp}{\textsc{qcqp}}{Quadratically Constrained Quadratic Program} 
\newacronym{rapidpf}{rapid\textsc{pf}}{rapid prototyping for distributed Power Flow}
\newacronym{admm}{\textsc{admm}}{Alternating Direction Method of Multipliers}
\newacronym{dqa}{\textsc{dqa}}{Diagonal Quadratic Approximation}
\newacronym{aladin}{\textsc{aladin}}{Augmented Lagrangian based Alternating Direction Inexact Newton method}
\newacronym{aladincor}{\textsc{aladin}-\textsc{cor}}{\acrshort{aladin} with second-order correction}
\newacronym{aladinstd}{\textsc{aladin}-\textsc{std}}{standard \acrshort{aladin}}
\newacronym{hgd}{\textsc{hgd}}{heterogeneous decomposition}
\newacronym{dcc}{\textsc{dcc}}{Distribution-Cost-Correction}
\newacronym{rmse}{\textsc{rmse}}{Root-Mean-Square Error}
\newacronym{pcc}{PPC}{point of common coupling}
\newacronym{vsc}{VSC}{voltage source converter}
\newacronym{mtdc}{MTDC}{multiterminal high voltage direct current}
\newacronym{hvdc}{HVDC}{high voltage direct current }
\newacronym{hvac}{HVAC}{high voltage alternating current }
\newacronym{lcc}{LCC}{line commutated converters}
\newacronym{lcc-hvdc}{LCC-HVDC}{\acrfull{lcc} based multiterminal high voltage direct current}
\newacronym{vsc-mtdc}{VSC-MTDC}{\acrfull{vsc} based multiterminal high voltage direct current}
\newacronym{igbt}{IGBT}{insulated gate bipolar transistor}
\newacronym{vsc-hvdc}{VSC-HVDC}{\acrfull{vsc} based high voltage direct current}
\newacronym{opf}{OPF}{optimal power flow}
\newacronym{bim}{BIM}{bus injection model}
\newacronym{bfm}{BFM}{branch flow model}

\newacronym{pu}{p.u.}{per-unit system}
\newacronym{pwm}{PWM}{pulse-width modulation}
\newacronym{licq}{LICQ}{linear independence constraint qualification} 
\newacronym{sosc}{SOSC}{second order sufficient condition} 
\newacronym{scc}{SCC}{strict complementarity conditions} 
\newacronym{kkt}{KKT}{Karush–Kuhn–Tucker}

\renewcommand*{\proofname}{Proof}

\newcommand{\abs}[1]{\left|#1\right|}

\newcommand{\diag}[1]{\mathrm{diag}\left(#1\right)}

\newcommand{\squarebrackets}[1]{\left[#1\right]}

\newcommand{\lin}{{\textsc{l}}}
\newcommand{\non}{{\textsc{n}}}

\usepackage{enumitem} 

\newif\ifshowrevisions
\showrevisionstrue

\DeclareRobustCommand{\rev}[1]{{\ifshowrevisions\color{black}\fi#1}}
\newenvironment{revblock}{\begingroup\ifshowrevisions\color{black}\fi}{\endgroup}

\newcommand{\ca}{\textsc{ca}}
\newcommand{\anc}{\textsc{a}}
\newcommand{\seg}{\textsc{s}}

\newcommand{\R}{\mathbb R}
\newcommand{\C}{\mathbb C}

\newcommand{\rr}{\mathscr r}
\newcommand{\xx}{\mathscr x}
\newcommand{\zz}{\mathscr z}

\newcommand{\Zb}{\mathscr{Z}_{\mathrm{b}}}

\makeatletter
\DeclareFontEncoding{LS1}{}{}
\DeclareFontSubstitution{LS1}{stix}{m}{n}
\DeclareMathAlphabet{\mathscr}{LS1}{stixscr}{m}{n}
\makeatother

\hypersetup{
    linkcolor=red!60!black,  
    citecolor=blue!60!black, 
    urlcolor=blue!60!black,  
}

\title{\centering DistFlow and Its Conic Relaxation: Regularity, Monotonicity, \\and Hybrid Exactness via Graph Decomposition}
\author[1]{\fnm{Xinliang} \sur{Dai}}

\affil[1]{\orgdiv{Andlinger Center for Energy and the Environment}, \orgname{Princeton University}
}

\abstract{\normalsize
Classical exactness conditions for conic relaxations of the DistFlow model
rely, among other assumptions, either on all nodal injections being adjustable
or on a network-wide path condition. In practice, these conditions may not hold
uniformly across a feeder: controllable distributed energy resources (DERs) sit
at some buses and fixed loads at others, while reverse flows from DER-rich
regions can violate the path condition. This paper certifies exactness region
by region and composes the regional certificates, accounting for perturbations
that propagate downstream through voltages. We first introduce a subtree
voltage stability region, in which the network and all compatible subtrees
remain regular. Within it, two endpoint sign conditions, one relaxing the
nonnegative-flow assumption and one implied by the classical path condition,
yield monotone responses of voltages, currents, and losses to root voltages.
Using duality, the implicit function theorem, and sign properties of inverse
sensitivity matrices, we recast both classical arguments as individual sign
certificates on branch multipliers. This turns
exactness into a branchwise property that can be certified region by region.
A partition rule then splits the network into dispatchable and passive
subtrees, and a hybrid certificate combines their local conditions into
exactness of the full-network relaxation, under additional operating-limit and
objective conditions. In particular, placing controllable DERs at both ends of
a branch with large reverse flow exempts that branch from the passive sign
conditions.
}

\begin{document}
\setlength\abovedisplayskip{4pt}
\setlength\belowdisplayskip{4pt}

\maketitle

\hypersetup{
            linkcolor=blue!80!black,  
            citecolor=green!50!black, 
            urlcolor=blue!80!black,   
            anchorcolor=blue
}

\section{Introduction}\label{sec:introduction}

The increasing penetration of distributed energy resources (DERs), such as
rooftop photovoltaics, battery storage, and electric vehicle charging, is
turning distribution networks into actively operated systems~\cite{karthikeyan2019activeDistr,dutta2020topology}. The resulting variability and spatial distribution of
power injections motivate tractable methods for AC optimal power flow
(OPF).  However, these problems are nonconvex: AC feasibility is NP-hard even on tree networks~\cite{lehmann2015ac} and strongly NP-hard in general~\cite{bienstock2019strong}.

Convex relaxations address this nonconvexity from below. The DistFlow
equations~\cite{baran1989optimal1,baran1989optimal2} expose the tree structure
of a radial network and confine the nonconvexity to one quadratic equality per
branch.
Relaxing each equality to \rev{an} inequality yields a conic relaxation~\cite{jabr2006radial,farivar2013branch}. Its optimal value bounds the OPF cost from below, and a relaxed optimizer at which every cone constraint is tight is globally optimal for the original OPF. However, \rev{convex} relaxations may not be exact~\cite{molzahn2014investigation,lesieutre2011examining}.
Exactness must be guaranteed by additional conditions on the controls, the
network parameters, the operating limits, and the
objective; see~\cite{low2014convex1,LowStevenH2014,molzahn2019survey,steven2025PSA} for \rev{surveys}.

Two classical families of such conditions rest on different assumptions about \rev{the available controls of} nodal injections. The first, among other assumptions, \rev{requires} a fully \textit{\rev{dispatchable}} system, i.e., all nodal injections are controllable~\cite{farivar2011inverter,farivar2013branch}.  The second combines a path condition~\cite[Thm.~1, C1]{gan2014exact}, \rev{satisfied for instance under constant, nondecreasing, or nonincreasing resistance-to-reactance ratios from the slack bus toward the leaves}~\cite[Thm.~4]{gan2014exact}, with voltage upper bound restrictions. This approach \rev{also covers} \textit{passive} systems, \rev{in which} nodal injections are \rev{fixed}. Both approaches use perturbation arguments at a relaxed solution to show that the conic inequalities are tight.

In practice, controllable DERs are installed at some buses, while the
remaining buses have fixed injections. We call such networks
\emph{partially dispatchable}. The classical conditions then need not hold
uniformly across a feeder: the dispatchability assumption fails at every
fixed-injection bus, and a network-wide path condition may be violated by a
single region with large DER injections and reverse power flow. Nevertheless, each
condition may hold locally. This \rev{suggests} certifying exactness
region by region.

The regions, however, remain electrically coupled, and the two mechanisms couple
in different ways. The dispatchable perturbation tightens one branch while
holding every voltage fixed~\cite[Thm.~1]{farivar2013branch}, so its effect stays on that branch and its incident
buses. The passive perturbation~\cite[Thm.~1]{gan2014exact} cannot avoid changing voltages. If the region
feeds other regions, it shifts their root voltages, which in turn changes their
currents, their losses, and the power they draw across the interface. A regional
argument is therefore valid for the full network only if this propagated
response
\begin{itemize}
  \item is well defined, which requires the power flow equations of every
  subtree, with its own local slack bus, to be regular;
  \item has signs that preserve the cost decrease, which requires voltages,
  currents, or losses to respond monotonically to the root voltage; and
  \item can be absorbed by controllable interface injections without leaving
  the operating set.
\end{itemize}
Existing regularity and sensitivity results provide part of this foundation. In particular,
Aolaritei et al.~\cite{aolaritei2017distributed,aolaritei2018hierarchical}
use a continuation argument from the flat solution to establish a
nonsingular $M$-matrix property of the reduced power-flow Jacobian
under nonnegative branch flows. Their argument requires the relevant
sign structure along the continuation path, not merely at its
endpoint. Current-based conditions certify complex Jacobian nonsingularity
pointwise~\cite{wang2016necessary,cui2018new}, fixed-point analyses establish
existence and uniqueness within specified solution
domains~\cite{bolognani2016existence,bernstein2018loadflow,cui2025securing},
solvability and high-voltage solution properties are known for lossless or
structured radial networks~\cite{simpsonporco2018solvability2,dvijotham2017high},
and sensitivity coefficients of voltages and currents in radial networks can be
computed efficiently~\cite{christakou2013efficient}.
However, none of these results establishes that all compatible
subtrees, including those with an interior local slack bus, remain regular, nor does any give the signs of voltage, current, and loss responses to a root-voltage change with injections or receiving-end branch powers held fixed.

By exploring the regularity and monotonicity of the DistFlow equations, this
paper develops a compositional exactness certificate for partially
dispatchable radial networks. The main contributions are threefold.
\begin{enumerate}[label=(\roman*)]
  \item We introduce a subtree voltage stability region, in which the network
  and all compatible subtrees remain regular, and show that
  it contains the regular sets certified
  in~\cite{aolaritei2018hierarchical,wang2016necessary,cui2018new}. Within it,
  we establish \rev{nonnegative inverses} of sensitivity matrices, and thus monotone responses to the slack voltage under two endpoint sign
  conditions: one relaxes the nonnegative flow assumption
  of~\cite{aolaritei2018hierarchical}, and the other is implied by condition C1
  of~\cite[Thm.~1]{gan2014exact}.
  \item Following the positive multiplier argument~\cite{gan2012exactness}, we recast the perturbation mechanisms of~\cite[Thm.~1]{farivar2013branch}
  and~\cite[Thm.~1]{gan2014exact} as individual sign
certificates on branch multipliers via duality, the implicit function theorem, and sign properties of inverse sensitivity matrices. This turns
exactness into a branchwise property that can be certified region by region.
  In particular, for passive networks, we give the multipliers in closed form, and  this
  expression bounds them from below when the slack voltage or injections vary.
  \item We give a partition rule that splits the network into dispatchable and
  passive subtrees, and compose the local directions into a hybrid exactness
  certificate: passive directions are extended through descendant subtrees,
  and controllable interface injections restore the power balances. If each
  passive subtree satisfies a local \rev{sign} condition and the completed
  directions are tangent to the operating set \rev{and decrease the objective}, all
  multipliers are strictly positive, which yields global optimality and
  exactness of the relaxation.
\end{enumerate}
The rest of this paper is organized as follows: 
Section~\ref{sec::model} presents the problem formulation for the DistFlow
model and its conic relaxation. Section~\ref{sec:regularity-sensitivity}
introduces the subtree voltage stability region and establishes monotone
responses to the slack voltage under the two sign conditions.
Section~\ref{sec:exactness-certificates} recasts the classical certificates
as multiplier sign tests for fully dispatchable and passive networks.
Section~\ref{sec:hybrid-exactness} partitions a partially dispatchable
network, composes the local directions into the hybrid exactness certificate,
and discusses the placement of controllable DERs.
Section~\ref{sec:conclusion} concludes the paper.

\section{Problem Formulation}\label{sec::model}
\subsection{Notations}
All vector and matrix inequalities are componentwise. The orthants $\R^n_+$, $\R^n_-$, $\R^n_{++}$, and $\R^n_{--}$ contain vectors with nonnegative, nonpositive, positive, and negative entries, respectively. $\R^{m\times n}_+$, $\R^{m\times n}_-$, $\R^{m\times n}_{++}$, and $\R^{m\times n}_{--}$ impose the corresponding entrywise signs on matrices. We write $(x,y):=[x^\top,y^\top]^\top$, $[A]_{ij}$ for a matrix entry, $A[\mathcal I,\mathcal J]$ for the submatrix with rows $\mathcal I$ and columns $\mathcal J$, and $\diag{x}$ for the diagonal matrix with diagonal $x$. The vectors $\mathbf1$ and $\mathbf e_j$ denote the all-ones and canonical basis vectors. The spectral radius is $\rho(A)$; $\mathbf j:=\sqrt{-1}$. For differentiable functions, $\partial_xg$ is the Jacobian, $\nabla f$ is a column gradient, and $\mathrm D_y$ includes dependence through an implicitly defined state. A $Z$-matrix has nonpositive off-diagonal entries. A nonsingular $M$-matrix is a $Z$-matrix whose leading principal minors are all positive. Equivalently, a $Z$-matrix with nonnegative inverse~\cite[Thm.~6.2.3]{berman1994nonnegative}.

\subsection{Tree Graph Orientation}\label{sec::graph}
Let $\mathcal G=(\overline{\mathcal N},\mathcal L)$ be a tree rooted at slack bus $0$, with $\overline{\mathcal N}=\{0\}\cup\mathcal N$ and $\mathcal N=\{1,\ldots,n\}$. Branches are oriented away from $0$, and branch $j$ terminates at bus $j$, whose parent is $\pi(j)$. \rev{Buses are indexed in topological order, so that $\pi(j)<j$ for every $j\in\mathcal N$; hence every prefix $\{0,\ldots,k\}$ induces a connected subtree.}

Let $\mathcal{S}(\mathcal G)$ index the connected subtrees with \rev{at least} two buses. A subtree $\kappa$ has a local slack bus $s_\kappa$ nearest to $0$, non-slack bus index set $\mathcal N_\kappa$, \rev{branch set $\mathcal L_\kappa$,} and bus set $\overline{\mathcal N}_\kappa=\{s_\kappa\}\cup\mathcal N_\kappa$. It retains the original orientation. The subset rooted at $0$ is $\mathcal{S}_0(\mathcal G)$.

Following~\cite[Chp~4.3]{steven2025PSA}, we introduce the incidence matrices for $\mathcal G$:
\begin{defi}[Incidence]\label{def:ancestor}\label{def:incidence}
The incidence matrix is $$C=C_f-C_t\in\R^{(n+1)\times n},$$ with $+1$ at each sending bus and $-1$ at its receiving bus. Removing the slack row of
\[
C=\begin{bmatrix}
  c_0^\top\\
  \hat{C}
\end{bmatrix}
\]
gives the reduced incidence matrix
\begin{equation}\label{eq::reducedIncidence}
\hat C= \hat C_f-
\hat C_t=\hat C_f-\mathbf I_n\in\R^{n\times n}.
\end{equation}
\end{defi}
We denote the set of strict non-slack ancestors of $j$ as $\mathcal N_j^{\anc}$, the ancestors including $j$ as $\mathcal N_j^{\anc+}:=\mathcal N_j^{\anc}\cup\{j\}$, and the common ancestors of $i,j$ as $\mathcal N_{ij}^{\ca}:=\mathcal N_i^{\anc}\cap\mathcal N_j^{\anc}$.
\begin{lemm}[Reduced Incidence]\label{lem:incidence-matrix}\label{lem:path-structure}\label{lem:path-CfC}
The matrix $-\hat C$ is a nonsingular $M$-matrix with $$\det(\hat C)=(-1)^n,\quad\mathrm{and}\quad [\hat C^{-1}]_{ij}=
 \begin{dcases}-1,&i\in\mathcal N_j^{\anc+},\\0,&\text{otherwise}.
 \end{dcases}$$
Additionally,
\begin{equation}\label{eq::Cc0}
\hat C^{-\top}c_0 = - \mathbf 1, \;
\hat C_f^\top\mathbf1+c_0=\mathbf1,\;
\hat C^\top\mathbf1+c_0=0.
\end{equation}
\end{lemm}
\begin{proof}
Eq.~\eqref{eq::Cc0} comes from \cite[Thm.~4.1]{steven2025PSA} with Eq.~\eqref{eq::reducedIncidence}.
\rev{Under the topological ordering $\pi(j)<j$, $\hat C_f$ is strictly upper triangular, so $\hat C=\hat C_f-\mathbf I_n$ is upper triangular with $\det(\hat C)=(-1)^n$, and $-\hat C^{-1}=\sum_{k\ge0}\hat C_f^k\in\R^{n\times n}_+$. Since $[\hat C_f^k]_{ij}=1$ exactly when $i$ is the $k$th ancestor of $j$, this sum gives the stated path entries.} Since \rev{the} off-diagonal entries of $-\hat C$ are nonpositive, $-\hat C$ is \rev{a} nonsingular $M$-matrix by~\cite[Thm.~6.2.3.N38]{berman1994nonnegative}.
\end{proof}
\rev{Using the path structure of $\hat C^{-1}$, we introduce the strict ancestor matrix}
\begin{equation}\label{eq::ancestorMatrix}
 \Gamma:=-\hat C_f\hat C^{-1}=-\mathbf I-\hat C^{-1},
 \rev{\qquad [\Gamma]_{ij}=
 \begin{cases}1,&i\in\mathcal N_j^{\anc},\\0,&\text{otherwise},\end{cases}}
\end{equation}
\rev{so that $\Gamma\in\R^{n\times n}_+$ is strictly upper triangular.}

\subsection{DistFlow Equations}\label{sec:distflow}\label{sec::distflow}
We use a balanced single-phase equivalent with series branch impedances and no shunts or tap-changing devices. For branch $j$, let $\zz_j=\rr_j+\mathbf j\xx_j$ be its impedance, where $\rr_j,\xx_j\ge0$ and $|\zz_j|>0$. The receiving-end complex power is $P_j+\mathbf jQ_j$, and $\ell_j$ is the squared current magnitude. At bus $j$, $u_j$ is the squared voltage magnitude and $p_j+\mathbf jq_j$ is the net complex power injection, positive for generation. The slack voltage $u_s:=u_0=|V_0|^2>0$ is fixed unless explicitly varied in a sensitivity calculation.

Following~\cite{baran1989optimal1,baran1989optimal2}, the receiving-end DistFlow equations can be written as
\begin{subequations}\label{eq::distflow}
\begin{align}
 p_j&=\sum_{k:\pi(k)=j}(P_k+\rr_k\ell_k)-P_j&\forall j\in\mathcal N,\\
 q_j&=\sum_{k:\pi(k)=j}(Q_k+\xx_k\ell_k)-Q_j&\forall j\in\mathcal N,\\
 u_j&=u_{\pi(j)}\hspace{-2pt} -2(\rr_jP_j+\xx_jQ_j)\hspace{-1pt}-\hspace{-1pt}|\zz_j|^2\ell_j&\forall j\in\mathcal N,
\\
\hspace{-6pt} u_j\ell_j&=P_j^2+Q_j^2&\forall j\in\mathcal N.\label{eq:DistFlow-nonlinear}
\end{align}
\end{subequations}
For the Jacobian analysis, collect these quantities in vectors $z:=(P,Q,u,\ell,p,q)\in\R^{6n}$ and define
\begin{align*}
 \mathbf P&=\diag{P},\;&\mathbf U=\diag{u},\quad&\mathbf X=\diag{\xx},\\
 \mathbf Q&=\diag{Q},\;&\mathbf L=\diag{\ell},\quad&\mathbf R=\diag{\rr},
 \rev{\quad\mathbf Z=\diag{\abs{\zz}}.}
\end{align*}
Then, the DistFlow equations~\eqref{eq::distflow} can be written in compact form as
\begin{subequations}\label{eq:DistFlow-matrix}
\begin{align}
 0&=\hat CP+\hat C_f\mathbf R\ell-p,\label{eq:DistFlow-matrix:P}\\
 0&=\hat CQ+\hat C_f\mathbf X\ell-q,\label{eq:DistFlow-matrix:Q}\\
 0&=2(\mathbf RP+\mathbf XQ)+\mathbf Z^2\ell-\hat C^\top u-c_0u_s,\label{eq:DistFlow-matrix:U}\\
 0&=\mathbf PP+\mathbf QQ-\mathbf U\ell.\label{eq:DistFlow-matrix:nonlinear}
\end{align}
\end{subequations}
The slack injections are
\begin{subequations}\label{eq::slack::balance}
\begin{align}
 p_0&=\sum_{j:\pi(j)=0}(P_j+\rr_j\ell_j),\\
 q_0&=\sum_{j:\pi(j)=0}(Q_j+\xx_j\ell_j).
\end{align}
\end{subequations}
Summing \rev{\eqref{eq:DistFlow-matrix:P} and~\eqref{eq:DistFlow-matrix:Q} over all buses and using~\eqref{eq::Cc0}} yields the power balance of the system:
\begin{equation}\label{eq::slack::injections}
 p_0=\rr^\top\ell-\mathbf1^\top p,\qquad
 q_0=\xx^\top\ell-\mathbf1^\top q.
\end{equation}
\subsection{OPF and Its Conic Relaxation}\label{sec:opf-distflow-soc}
We collect the operating limits on $z$ in $\mathcal Z$ and retain the DistFlow equations~\eqref{eq:DistFlow-matrix} as explicit constraints. Let $f$ be continuously differentiable on a neighborhood of the operating set. We use the closed convex box
\begin{equation}\label{eq:operating-set}
 \mathcal Z:=\{z\in\R^{6n}:z_\text{\normalfont\scshape min}\le z\le z_\text{\normalfont\scshape max}\},
\end{equation}
with strictly positive lower voltage bounds. \rev{Fixed injections are represented by coincident lower and upper bounds.}

The original OPF problem is
$
\min_{z\in\mathcal Z}\ f(z),$ subject to $\quad g(z)=0,
$
or, separating the affine parts~\eqref{eq:DistFlow-matrix:P}-\eqref{eq:DistFlow-matrix:U} and nonlinear part~\eqref{eq:DistFlow-matrix:nonlinear} as
\begin{subequations}\label{eq:opf-distflow-split}
\begin{align}
 \min_{z\in\mathcal Z}\quad&f(z)\\
 \text{s.t.}\quad&g_\lin(z)=0\quad|\;\lambda_\lin,\\
 &g_\non(z)=0\quad|\;\lambda_\non.
\end{align}
\end{subequations}
Here $\lambda_\lin\in\R^{3n}$ and $\lambda_\non\in\R^n$ are the associated multipliers. Relaxing the quadratic equalities~\eqref{eq:DistFlow-matrix:nonlinear} gives
\begin{subequations}\label{eq:opf-soc}
\begin{align}
 \min_{z\in\mathcal Z}\quad&f(z)\\
 \text{s.t.}\quad&g_\lin(z)=0\quad|\;\lambda_\lin,\\
 &g_\non(z)\leq0 \quad|\;\lambda_\non.\label{eq:opf-soc::n}
\end{align}
\end{subequations}
Each relaxed constraint in~\eqref{eq:opf-soc::n} is the rotated-cone inequality $P_j^2+Q_j^2\le u_j\ell_j$, with $u_j>0$ and $\ell_j\ge0$.
Consequently the relaxed feasible set is convex.
For \rev{the} objective function, we set
\[
 f(z)=\varphi(u,\ell,p,q;p_0(z),q_0(z)),
\]
where $p_0,q_0$ are dependent variables \rev{given by}~\eqref{eq::slack::injections}.

\rev{The relaxed problem~\eqref{eq:opf-soc} is then convex under the following standing assumption, which holds throughout the paper.}
\begin{ass}[Convex objective and operating set]\label{ass:convex-exactness}
The operating set $\mathcal Z$ is closed and convex, with positive lower voltage bounds. The objective $f$ is continuously differentiable and convex on $\mathcal Z$.
\end{ass}

\rev{We analyze the exactness of the relaxed problem~\eqref{eq:opf-soc} in the following all-optimizer sense.}
\begin{defi}[Exactness~{\cite[Def.~1]{gan2014exact}}]\label{def:exact}
Assuming the relaxed optimum of~\eqref{eq:opf-soc} is attained, the \rev{relaxation~\eqref{eq:opf-soc}} is \emph{exact} if every optimal solution satisfies $g_\non(z)=0$.
\end{defi}
\rev{The certificates developed in Sections~\ref{sec:exactness-certificates} and~\ref{sec:hybrid-exactness} are evaluated at a KKT point of the original problem~\eqref{eq:opf-distflow-split}. They rely on two properties of the DistFlow equations, regularity and monotonicity with respect to the slack voltage, which we establish first.}

\section{Regularity and Sensitivity of DistFlow}
\label{sec:regularity-sensitivity}
The exactness analysis relies on the regularity and sensitivity of the DistFlow equations~\eqref{eq:DistFlow-matrix} with \rev{their} Jacobian
\begin{equation*}
\mathrm Dg(z;u_s)
=
\left[\begin{array}{ccccrrc}
 \hat C&\mathbf0_n&\mathbf0_n&\hat C_f\mathbf R&-\mathbf I_n&\mathbf0_n&0\\
 \mathbf0_n&\hat C&\mathbf0_n&\hat C_f\mathbf X&\mathbf0_n&-\mathbf I_n&0\\
 2\mathbf R&2\mathbf X&-\hat C^\top&\hspace{10pt}\mathbf Z^2&\mathbf0_n&\mathbf0_n&-c_0\\
 2\mathbf P&2\mathbf Q&-\mathbf L\hspace{4pt}&-\mathbf U&\mathbf0_n&\mathbf0_n&0
\end{array}\right],
\end{equation*}
with respect to $(P,Q,u,\ell,p,q,u_s) \in\R^{6n+1}.$ \rev{Section~\ref{sec::nonsingularity} introduces a regularity notion that covers the network together with all of its connected subtrees. Section~\ref{sec::condensed} eliminates the affine equations and reduces regularity to a single $n\times n$ current Jacobian $\mathbf M_\ell$. Sections~\ref{sec::monotonicity::ul} and~\ref{sec:gan-matrix-certificate} give two endpoint sign conditions under which inverses built from $\mathbf M_\ell$ are nonnegative, and Section~\ref{sec:slack-voltage-monotonicity} turns these sign properties into monotone responses to the slack voltage.}
\subsection{Constraint Regularity and Subtree Voltage Stability}
\label{sec::nonsingularity}
The partial Jacobian of the DistFlow constraint $g$ with respect to $(P,Q,u,\ell)$ is
\[
J(z;u_s):=\allowbreak\partial_{(P,Q,u,\ell)}g(z;u_s)=\begin{bmatrix}
 \hat C&0&0&\hat C_f\mathbf R\\
 0&\hat C&0&\hat C_f\mathbf X\\
 2\mathbf R&2\mathbf X&-\hat C^\top&\hspace{10pt}\mathbf Z^2\\
 2\mathbf P&2\mathbf Q&-\mathbf L\hspace{4pt}&-\mathbf U
 \end{bmatrix}.
\]
Its nonsingularity gives a locally unique smooth solution map with respect to $(p,q;u_s)$. Define
\begin{equation}\label{eq::Rvs}
 \mathcal R:=\{z\,|\,g(z;u_s)=0,\ u\in\R^n_{++},\ \det J(z;u_s)\ne0\}.
\end{equation}
Eq.~\eqref{eq:DistFlow-matrix:nonlinear} implies $\ell\in\R^n_{+}$, and the flat solution
$z^0:=(\mathbf0_{2n},u_s\mathbf1_n,\mathbf0_{3n})$ belongs to $\mathcal R$. Following~\cite[Chp.~7]{van2007voltage}~\cite{aolaritei2018hierarchical}, we adopt the  standard definition for voltage stability:
\begin{defi}[Voltage stability region]
\label{def:voltage-stability}
The region $$\mathcal R_{\mathrm{vs}}:=\operatorname{comp}_{z^0}(\mathcal R),$$ is the path-connected component containing $z^0$.
\end{defi}
In the present paper, we are also interested in the constraint regularity of \rev{the} subtrees. \rev{To this end}, we introduce the following compatible operating point for each connected subtree $\kappa\in\mathcal{S}(\mathcal G)$ by retaining the \rev{variables} $(P_j,Q_j,u_j,\ell_j),$ for all $j\in\mathcal N_\kappa$ with $n_\kappa:=\abs{\mathcal N_\kappa}$.
\begin{defi}[Compatible subtree operating point]
\label{def:compatible-subtree-operating-point}
For $\kappa\in\mathcal{S}(\mathcal G)$, restrict each $a\in\{P,Q,u,\ell\}$ to $a_\kappa:=a[\mathcal N_\kappa]$, use the root voltage $u_{s_\kappa}$, and define
\begin{subequations}\label{eq:compatible-subtree-injections}
\begin{align}
 p_i^\kappa&:=p_i-\sum_{\substack{j:\pi(j)=i\\j\notin\mathcal N_\kappa}}(P_j+\rr_j\ell_j),
 \qquad i\in\mathcal N_\kappa\\
 q_i^\kappa&:=q_i-\sum_{\substack{j:\pi(j)=i\\j\notin\mathcal N_\kappa}}(Q_j+\xx_j\ell_j),
 \qquad i\in\mathcal N_\kappa.
\end{align}
\end{subequations}
The compatible point is $(z_\kappa;u_{s_\kappa})$, where
\[
 z_\kappa:=(P_\kappa,Q_\kappa,u_\kappa,\ell_\kappa,p^\kappa,q^\kappa)\in\R^{6n_\kappa}.
\]
\end{defi}

If $g(z;u_s)=0$, these adjustments give $g_\kappa(z_\kappa;u_{s_\kappa})=0$ on the retained branches, with its local slack bus $s_\kappa$. \rev{In addition to the conventional regularity~\eqref{eq::Rvs}, subtree voltage stability requires} the local partial Jacobian
$$J_\kappa
:=\partial_{(P_\kappa,Q_\kappa,u_\kappa,\ell_\kappa)}g_\kappa(z_\kappa,u_{s_\kappa})\in\R^{4n_\kappa\times 4n_\kappa},$$
\rev{to be} nonsingular for each $\kappa \in\mathcal{S}(\mathcal G)$.
\begin{defi}[Subtree voltage stability region]
\label{def:strong-voltage-regularity}
The subtree voltage stability region is
\[
 \mathcal R_{\mathrm{svs}}:=\operatorname{comp}_{z^0}
 \left\{z\in\mathcal R\,|\, \det(J_\kappa)
 \ne0,\quad\forall \kappa\in\mathcal{S}(\mathcal G)
\right\}.
\]
\end{defi}
Thus $\mathcal R_{\mathrm{svs}}\subseteq\mathcal R_{\mathrm{vs}}\subseteq\mathcal R$. Membership requires one common path from $z^0$ on which every compatible subtree remains regular. Local injections and root voltages may vary along that path, and fixing them when evaluating $J_\kappa$ is a derivative convention. \rev{The region is inherited by the subtrees themselves, which is what the regional arguments of Section~\ref{sec:hybrid-exactness} require.}

\begin{lemm}[Inheritance of $\mathcal R_{\mathrm{svs}}$]
\label{lem:strong-subtree-inheritance}
If $z\in\mathcal R_{\mathrm{svs}}$, then every compatible subtree point satisfies $z_\kappa\in\mathcal R_{\mathrm{svs},\kappa}$, where the local region is defined with slack voltage fixed at $u_{s_\kappa}$.
\end{lemm}
\begin{proof}
Restrict a common regular path $z(t)$ to $\kappa$. Compatible restriction is associative, so $z_\kappa(t)$ and all its compatible subtrees remain regular. Multiply every coordinate by $a(t):=u_{s_\kappa}(1)/u_{s_\kappa}(t)>0$. This fixes the local slack voltage at its endpoint value. The affine residuals scale by $a(t)$, the nonlinear residuals by $a(t)^2$, and only the final block row of each fixed-injection Jacobian scales, by $a(t)$. Nonsingularity is preserved. The scaled path joins the local flat solution to $z_\kappa$, proving the claim.
\end{proof}
Membership in $\mathcal R_{\mathrm{svs}}$ can be certified by a pointwise current condition. Wang et al.~\cite[Thm.~2]{wang2016necessary} show, in complex form, that the power flow Jacobian of a radial network is nonsingular whenever every bus voltage magnitude exceeds the sum of the equivalent voltage drops produced by the nodal injection currents. Cui and Sun~\cite[Thm.~1]{cui2018new} give the same criterion in real variables and show that the set it defines is connected to the no-load solution~\cite[Thm.~2]{cui2018new}. We restate the criterion in DistFlow variables. For a tree with series impedances, the bus impedance matrix is
\begin{equation}\label{eq:bus-impedance}
 \Zb:=\hat C^{-\top}\diag{\zz}\hat C^{-1}\in\C^{n\times n},\qquad
 [\Zb]_{ij}=\sum_{k\in\mathcal N_i^{\anc+}\cap\mathcal N_j^{\anc+}}\zz_k,
\end{equation}
whose $(i,j)$ entry sums the impedances on the common root path of $i$ and $j$. At a feasible point with $u\in\R^n_{++}$, bus $j$ draws the injection current magnitude $|s_j|/|V_j|$, which in DistFlow variables reads
\begin{equation}\label{eq:injection-current}
 \iota_j:=\sqrt{(p_j^2+q_j^2)/u_j},\qquad j\in\mathcal N.
\end{equation}
\begin{prop}[A sufficient certificate for subtree voltage stability]
\label{thm:wang-strong-regularity}
Let $z$ satisfy $g(z;u_s)=0$ with $u\in\R^n_{++}$. If
\begin{equation}\label{eq:wang-condition}
 \sqrt{u_i}>\sum_{j\in\mathcal N}\abs{[\Zb]_{ij}}\,\iota_j,
 \qquad \forall i\in\mathcal N,
\end{equation}
then $z\in\mathcal R_{\mathrm{svs}}$.
\end{prop}
Since $|V_i|=\sqrt{u_i}$, condition~\eqref{eq:wang-condition} is the complement of the boundary condition~\cite[Eq.~(42)]{wang2016necessary} and coincides with~\cite[Eq.~(9)]{cui2018new}. It involves only $(u,p,q)$ and the line impedances, and for fixed injections it is second-order cone representable in the squared voltages~\cite[Prop.~1]{cui2018new}. The pointwise criterion certifies nonsingularity of the full-network Jacobian at one point. Proposition~\ref{thm:wang-strong-regularity} adds that one current-scaling path from the flat solution keeps the criterion valid for the full network and for every compatible subtree, including those with an interior local slack bus. A proof sketch is given in Appendix~\ref{app:wang-strong-regularity}.

\subsection{Affine Elimination and Reduced Jacobians}\label{sec::condensed}
\rev{Nonsingularity of $J$ is most easily analyzed after eliminating the affine equations.} Partition $z=(x,y)$ with $$x=(P,Q,u)\quad\mathrm{and}\quad y=(\ell,p,q).$$
The affine Jacobian $\partial_x g_\lin\in\R^{3n\times 3n}$ is block-triangular, satisfying $\det(\partial_xg_\lin)=1$ as demonstrated in the proofs of~\cite[Thm.~1]{aolaritei2018hierarchical}~\cite[Prop.~1]{dai2024real}. Hence, the affine equations~\eqref{eq:DistFlow-matrix:P}-\eqref{eq:DistFlow-matrix:U} define a global map $x=x^\star(y;u_s)$.
Using the ancestor matrix~\eqref{eq::ancestorMatrix}
we define
\begin{subequations}
  \begin{align}
\hat{\mathbf Z}:=&\hspace{-2pt}-\mathbf Z^2\hspace{-1pt}-2\mathbf R\Gamma\mathbf R-2\mathbf X\Gamma\mathbf X\in\R^{n\times n}_{-},\\
\hat{\mathbf U}:=&\hspace{10pt}\mathbf U\hspace{2pt}-2\mathbf P\Gamma\mathbf R-2\mathbf Q\Gamma\mathbf X\in\R^{n\times n}.    \label{eq:hatU}
  \end{align}
\end{subequations}
Then, the \rev{derivative} of the map \rev{is}
\[
 \mathrm D x^\star(y;u_s)=
    \squarebrackets{\begin{array}{ccc;{2pt/2pt}c}
    \Gamma\mathbf R&\hat C^{-1}&\mathbf0_n&0\\
    \Gamma\mathbf X&\mathbf0_n&\hat C^{-1}&0\\
    -\hat{C}^{-\top}\hat{\mathbf Z}
    &2\hat C^{-\top}\mathbf R\hat C^{-1}
    &2\hat C^{-\top}\mathbf X\hat C^{-1}&\mathbf1
    \end{array}}
\]
with
\begin{equation}\label{eq:affine-elimination-slack}
 \partial_\ell x^\star(y;u_s)=
 \begin{bmatrix}\Gamma\mathbf R\\\Gamma\mathbf X\\-\hat C^{-\top}\hat{\mathbf Z}\end{bmatrix},
 \;
 \partial_{u_s}x^\star(y;u_s)=
 \begin{bmatrix}0\\0\\\mathbf1\end{bmatrix}.
\end{equation}
Define the reduced nonlinear residual
\[
 \hat g_\non(y;u_s):=g_\non(x^\star(y;u_s),y;u_s).
\]
Its derivatives give
\begin{align*}
\mathrm D_{(y,u_s)}\hat g_\non(y;u_s)=
-\begin{bmatrix}
 \mathbf M_\ell&
 \mathbf M_p&
 \mathbf M_q&
 m_s
\end{bmatrix}\in\R^{n\times(3n+1)},
\end{align*}
with
\begin{align*}
 \mathbf M_\ell
 &:=-\partial_\ell\hat g_\non(y;u_s)
 =\hat{\mathbf U}-\mathbf L \hat{C}^{-\top} \hat{\bf Z},\\
 \mathbf M_p
 &:=-\partial_p\hat g_\non(y;u_s)
 =2(\mathbf L\hat C^{-\top}\mathbf R-\mathbf P)\hat C^{-1},\\
 \mathbf M_q
 &:=-\partial_q\hat g_\non(y;u_s)
 =2(\mathbf L\hat C^{-\top}\mathbf X-\mathbf Q)\hat C^{-1},\\
 m_s&
 :=-\partial_{u_s}\hat g_\non(y;u_s)=
 \mathbf L\mathbf1=\ell,
\end{align*}
and, by Schur complement, we have
\begin{equation}\label{eq:subtree-det-reduced}
 \det(J)=(-1)^n\det(\mathbf M_\ell).
\end{equation}
Thus, nonsingularity of $J$ is equivalent to nonsingularity of $\mathbf M_\ell$ at a feasible point. For \rev{$z\in\mathcal R$ with $(p,q)$ fixed}, implicit differentiation gives
\begin{subequations}\label{eq::passive::dus}
\begin{align}
 \partial_{u_s}\ell^\star(p,q;u_s)&=-\mathbf M_\ell^{-1}\ell,\label{eq::dldus}\\
 \partial_{u_s}u^\star(p,q;u_s)&=\mathbf1+\mathbf K_u\ell,\label{eq::dudus}
\end{align}
\end{subequations}
with
\begin{align}\label{eq::Ku}
 \mathbf K_u:=\hat C^{-\top}\hat{\mathbf Z}\mathbf M_\ell^{-1}.
\end{align}
\rev{Regularity alone fixes no sign in~\eqref{eq::passive::dus}.} The following two certificates sign different parts of these formulas. The region $\mathcal F$ guarantees $\mathbf M_\ell^{-1}\in\R^{n\times n}_{+}$. The region $\mathcal C$ guarantees $\mathbf K_u\in\R^{n\times n}_{+}$ and nonnegative impedance-weighted inverse products.

\subsection{Current Certificate under Limited Reverse Flows}
\label{sec::monotonicity::ul}
The path structure gives an explicit test for $\mathbf M_\ell$ to be a $Z$-matrix. Combining that endpoint sign test with compatible-subtree regularity then yields a nonsingular $M$-matrix. \rev{We also show that the pathwise condition behind~\cite[Thm.~2]{aolaritei2018hierarchical}, once stated explicitly, is sufficient for subtree voltage stability.}

Set $\zeta_{kj}:=\rr_k\rr_j+\xx_k\xx_j\ge0$. Lemma~\ref{lem:path-CfC} gives
\begin{equation}\label{eq::hatU}
 [\hat{\mathbf U}]_{ij}=
 \begin{cases}
 u_j,&i=j,\\
 -2(\rr_jP_i+\xx_jQ_i),&i\in\mathcal N_j^{\anc},\\
 0,&\text{otherwise}.
 \end{cases}
\end{equation}
The same path expansion gives
\begin{equation}\label{eq:M-entry}
\begin{gathered}
[\mathbf M_\ell]_{ij}=
\begin{cases}
 u_j-\ell_j\Big(|\zz_j|^2+
 2\displaystyle\sum_{k\in\mathcal N_j^{\anc}}\zeta_{kj}\Big),&i=j,\\[1em]
 -2\Big(\rr_jP_i+\xx_jQ_i+
 \ell_i\displaystyle\sum_{k\in\mathcal N_i^{\anc+}}\zeta_{kj}\Big),
 &i\in\mathcal N_j^{\anc},\\[1em]
 -\ell_i\Big(|\zz_j|^2+
 2\displaystyle\sum_{k\in\mathcal N_j^{\anc}}\zeta_{kj}\Big),
 &j\in\mathcal N_i^{\anc},\\[1em]
 -2\ell_i\displaystyle\sum_{k\in\rev{\mathcal N_{ij}^{\ca}}}
 \zeta_{kj},&\text{otherwise}.
\end{cases}
\end{gathered}
\end{equation}
Only the ancestor off-diagonal entries $i\in\mathcal N^{\anc}_j$ can be positive.
\begin{ass}[Limited reverse flows]
\label{ass:ancestor-coupling}
For every $j\in\mathcal N$ and every $i\in\mathcal N_j^{\anc}$,
\begin{equation}\label{eq:ancestor-coupling}
 \rr_jP_i+\xx_jQ_i\ge
 -\ell_i\sum_{k\in\mathcal N_i^{\anc+}}\zeta_{kj}.
\end{equation}
\end{ass}
\begin{lemm}[$Z$-matrix Structure]\label{lem:Mell-Z}
For $\ell\in\R^n_{+}$, the off-diagonal entries of $\mathbf M_\ell$ are nonpositive if and only if Assumption~\ref{ass:ancestor-coupling} holds.
\end{lemm}
The nonnegative sending-end power assumption in~\cite[Ass.~1]{aolaritei2018hierarchical} is sufficient: substituting $P_i=P_i^s-\rr_i\ell_i$ and $Q_i=Q_i^s-\xx_i\ell_i$ into~\eqref{eq:ancestor-coupling} leaves
$$\rr_jP_i^s+\xx_jQ_i^s+\ell_i\sum_{k\in\mathcal N_i^{\anc}}\zeta_{kj}\ge0.$$
Define
\begin{equation}\label{eq::F}
 \mathcal F:=\{z:\text{Assumption~\ref{ass:ancestor-coupling} holds}\}.
\end{equation}
In the proof of~\cite[Thm.~2]{aolaritei2018hierarchical}, Aolaritei et al.\ use Perron--Frobenius theory
\cite[Ch.~8]{meyer2023matrix} to establish the nonsingular $M$-matrix property of the reduced power flow Jacobian under their nonnegative flow assumption \cite[Ass.~1]{aolaritei2018hierarchical}. To apply this continuation argument, regularity and the $Z$-matrix property must hold along the same continuous path from the flat solution $z^0$ to the operating point $z$.

Membership in $\mathcal R_{\mathrm{vs}}$ guarantees a regular path from $z^0$ to $z$. However, imposing
Assumption~\ref{ass:ancestor-coupling} only at the end point $z$ does not establish that this path remains in $\mathcal F$. We therefore make the continuation requirement explicit by
\begin{align*}
\operatorname{comp}_{z^0}(\mathcal R\cap\mathcal F)
     = \operatorname{comp}_{z^0} (\mathcal R_{\mathrm{vs}}\cap\mathcal F)
 \subseteq \mathcal R_{\mathrm{vs}}\cap\mathcal F.
\end{align*}
Thus, every point in $\operatorname{comp}_{z^0}(\mathcal R\cap\mathcal F)$ can be reached from $z^0$ along a path on which both regularity and the sign condition are preserved.
\begin{prop}[$M$-matrix via regular path]
\label{thm:continuation-M} If $$z\in\operatorname{comp}_{z^0}(\mathcal R\cap\mathcal F),$$ then $\mathbf M_\ell$ is a nonsingular $M$-matrix and $\mathbf M_\ell^{-1}\in\R^{n\times n}_+$.
\end{prop}
\begin{revblock}
\begin{proof}
Along a continuous path $z(\tau)\in\mathcal R\cap\mathcal F$ from $z^0$ to $z$, put $M=\mathbf M_\ell(z(\tau))$, $\Lambda=\diag{u(\tau)}$, and $K=\Lambda^{-1}(\Lambda-M)$. Eq.~\eqref{eq:M-entry} and Lemma~\ref{lem:Mell-Z} give $K\in\R^{n\times n}_{+}$, with $K(0)=0$. Nonsingularity of $M=\Lambda(\mathbf I-K)$ excludes $\rho(K)=1$ by Perron--Frobenius theory. Continuity therefore gives $\rho(K)<1$ throughout the path, proving the nonsingular $M$-matrix property~\cite[Thm.~6.2.3.N45]{berman1994nonnegative}.
\end{proof}
\end{revblock}
\rev{The argument needs $K\ge0$, hence Assumption~\ref{ass:ancestor-coupling}, at every point of the path. We next replace this pathwise sign requirement by compatible-subtree regularity and a sign test at the end point only.}

Fix a $\kappa\in\mathcal{S}(\mathcal G)$ and use the compatible operating point $(z_\kappa;u_{s_\kappa})$ from Definition~\ref{def:compatible-subtree-operating-point}. Distinguish the principal submatrix
\[
 A_\kappa:=\mathbf M_\ell[\mathcal N_\kappa,\mathcal N_\kappa]
\]
from the local reduced current Jacobian $\mathbf M_{\ell,\kappa}$, which is formed on the subtree $\mathcal G_\kappa$ with its own local slack bus $s_\kappa$. Comparing the entries in~\eqref{eq:M-entry} gives the exact correction for the upstream impedance removed when the slack changes:
\begin{equation}\label{eq:subtree-reduced-rank-one}
 \mathbf M_{\ell,\kappa}=A_\kappa+\ell_\kappa v_\kappa^\top,
\end{equation}
where
\[
 [v_\kappa]_j:=
 \begin{cases}
 0,&s_\kappa=0,\\[0.3em]
 2\displaystyle\sum_{k\in\mathcal N_{s_\kappa}^{\anc+}}
   \zeta_{kj},&s_\kappa\ne0,
 \end{cases}
 \qquad j\in\mathcal N_\kappa.
\]
For subtrees rooted at bus $0$, i.e., $\kappa\in\mathcal{S}_0(\mathcal G)$, we have $v_\kappa=0$. In particular, \rev{the} topological ordering in Section~\ref{sec::graph} makes each prefix $\{0,\ldots,k\}$ connected, and its local matrix equals the corresponding leading principal submatrix \rev{of} $\mathbf M_\ell$, i.e., $$\mathbf M_\ell^{[k]}:=\mathbf M_\ell[1\!:\!k,1\!:\!k]\rev{.}$$ Determinant continuity~\cite{aolaritei2017distributed,aolaritei2018hierarchical} on those subtrees rooted at $0$ yields the following signs.
\begin{lemm}[Leading principal minors of $\mathbf M_\ell$]
\label{lemm::positive::Ml}
If $z\in\mathcal R_{\mathrm{svs}}$, then
\[
 \det(\mathbf M_\ell^{[k]})>0,\qquad k=1,\ldots,n.
\]
\end{lemm}
\begin{proof}
\rev{At $z^0$, $\mathbf M_\ell=u_s\mathbf I_n$, so $\det(\mathbf M_\ell^{[k]})=u_s^k>0$. Along the common regular path, $\mathbf M_\ell^{[k]}$ is the local matrix $\mathbf M_{\ell,\kappa}$ of the compatible prefix subtree, whose determinant cannot vanish by~\eqref{eq:subtree-det-reduced}. Hence all leading principal minors remain positive.}
\end{proof}
\begin{prop}[$M$-matrix via positive principal minors]
\label{thm:strong-region-M}
If $$z\in\mathcal R_{\mathrm{svs}}\cap\mathcal F,$$ then $\mathbf M_\ell$ is a nonsingular $M$-matrix and $\mathbf M_\ell^{-1}\in\R^{n\times n}_+$.
\end{prop}
\begin{proof}
Since $\mathbf M_\ell$ is a $Z$-matrix (Lemma~\ref{lem:Mell-Z}) and all its leading principal minors are positive (Lemma~\ref{lemm::positive::Ml}), \rev{it} is a nonsingular $M$-matrix by~\cite[Thm.~6.2.3.E17]{berman1994nonnegative}.
\end{proof}
To compare the two certificates, we first record the subtree regularity implied by the $M$-matrix property itself.
\begin{lemm}[Subtree regularity under a nonsingular $M$-matrix]
\label{lem:Mell-subtree-regularity}
If $z\in\mathcal R$ and $\mathbf M_\ell$ is a nonsingular $M$-matrix, then
\begin{equation}\label{eq::subtree-regularity}
 \det(J_\kappa(z_\kappa;u_{s_\kappa}))\ne0,
 \qquad \forall \kappa\in\mathcal{S}(\mathcal G).
\end{equation}
\end{lemm}
\begin{proof}
For each $\kappa\in\mathcal{S}(\mathcal G)$, the principal submatrix $A_\kappa$ is also a nonsingular $M$-matrix, and hence $A_\kappa^{-1}\in\R^{n_\kappa\times n_\kappa}_+$ by~\cite[Thm 6.2.4]{berman1994nonnegative}. Moreover, $\ell_\kappa\in\R^{n_\kappa}_+$ and $v_\kappa\in\R^{n_\kappa}_+$. The matrix determinant lemma applied to \eqref{eq:subtree-reduced-rank-one} gives
\[
 \det(\mathbf M_{\ell,\kappa})
 =\det(A_\kappa)
       \bigl(1+v_\kappa^\top A_\kappa^{-1}\ell_\kappa\bigr)>0.
\]
The subtree version of~\eqref{eq:subtree-det-reduced} now gives~\eqref{eq::subtree-regularity}.
\end{proof}

\rev{Applying Proposition~\ref{thm:continuation-M} and Lemma~\ref{lem:Mell-subtree-regularity} at every point of a path in $\operatorname{comp}_{z^0}(\mathcal R\cap\mathcal F)$ shows that the path lies in the set of Definition~\ref{def:strong-voltage-regularity}, which gives the following inclusions.}
\begin{col}
\label{col:continuation-strong-region}
\begin{equation*}
\operatorname{comp}_{z^0}(\mathcal R\cap\mathcal F)
 \subseteq\mathcal R_{\mathrm{svs}}\cap\mathcal F
 \subseteq\mathcal R_{\mathrm{vs}}\cap\mathcal F.
\end{equation*}
\end{col}

Before proceeding, we record a necessary condition for  $z\in\mathcal R_{\mathrm{svs}}$.
\begin{lemm}
\label{lem:strong-branch-limit}
If $z\in\mathcal R_{\mathrm{svs}}$, then
\begin{equation*}
u_j-\abs{\zz_j}^2\ell_j>0,\qquad \forall j\in\mathcal N.
\end{equation*}
\end{lemm}
\begin{proof}
The compatible single-edge subtree $\{\pi(j),j\}$ has reduced Jacobian $u_j-|\zz_j|^2\ell_j$. It starts at $u_s>0$ and cannot vanish on the common path, by~\eqref{eq:subtree-det-reduced}.
\end{proof}
This branchwise condition \rev{coincides} with the sufficient power flow \rev{regularity condition} for a two-bus \rev{system}~\eqref{eq:wang-condition}.
\rev{It} also appears as Condition~C9.7 in the global optimality analysis via \rev{a} Lyapunov-like function in~\cite[Thm.~9.7]{steven2025PSA}.
\subsection{Voltage and Total-Loss Certificate}
\label{sec:gan-matrix-certificate}
\rev{The region $\mathcal F$ signs the whole inverse $\mathbf M_\ell^{-1}$. When reverse flows violate Assumption~\ref{ass:ancestor-coupling}, a weaker conclusion still suffices for the exactness arguments: we now sign the matrix $\mathbf K_u$ and the total-loss responses without requiring $\mathbf M_\ell^{-1}\in\R^{n\times n}_{+}$.}

The matrix $\hat{\mathbf U}$ is upper triangular by~\eqref{eq::hatU}. If $u\in\R^n_{++}$, we have
\begin{equation*}
 \det(\hat{\mathbf U})=\prod_{j\in\mathcal N}u_j>0.
\end{equation*}
Then, we \rev{factor} $$\mathbf M_\ell=\mathbf W\hat{\mathbf U},$$ where
\[
 \mathbf W:=\mathbf I-\mathbf L\hat C^{-\top}\hat{\mathbf Z}\hat{\mathbf U}^{-1}.
\]

\begin{ass}
\label{ass:voltage-ancestor-coupling}
At $z\in\mathcal R$, assume
\begin{subequations}\label{eq:T}
\begin{align}
 \mathbf T_r:=&-\hat C^{-1}\mathbf R\hat{\mathbf U}^{-1}\in\R^{n\times n}_+,\\
 \mathbf T_x:=&-\hat C^{-1}\mathbf X\hat{\mathbf U}^{-1}\in\R^{n\times n}_+.
\end{align}
\end{subequations}
\end{ass}
Set
\begin{equation}\label{eq::C}
\mathcal C:=\{z\,|\,\text{Assumption~\ref{ass:voltage-ancestor-coupling} holds}\}.
\end{equation}
Gan et al.'s path-product C1 condition~\cite[Thm.~1]{gan2014exact} implies \rev{Assumption}~\ref{ass:voltage-ancestor-coupling}. For completeness, we provide \rev{the} proof in Appendix~\ref{app:proof-gan-C1-ancestor-sums}.
\begin{lemm}
\label{lemma:gan-C1-ancestor-sums}
Suppose $\mathscr r,\mathscr x\in\R^{n}_{++}$ and  $z\in\mathcal R$ satisfy C1 from~\cite[Thm.~1]{gan2014exact}. Then $z\in\mathcal C$.
\end{lemm}

\begin{rem}\label{rem::sufficient::C}
Under the assumptions of Lemma~\ref{lemma:gan-C1-ancestor-sums}, each of the conditions (1)--(5) in~\cite[Thm.~4]{gan2014exact} implies $z\in\mathcal C$. These include conditions involving constant, nondecreasing, or nonincreasing resistance-to-reactance ratios from the slack bus toward the leaves. \rev{Since C1 is evaluated on injection and voltage bounds rather than at an operating point, it certifies $z\in\mathcal C$ at every feasible point within those bounds, before any OPF is solved.}
\end{rem}

\begin{rem}[\rev{Inheritance of $\mathcal C$}]\label{rem::C-inheritance}
The entries of $\hat C^{-1}$ and $\hat{\mathbf U}^{-1}$ are supported on
ancestor paths. For $i,j\in\mathcal N_\kappa$, the path from $i$ to $j$ lies in
the connected subtree $\kappa$, and compatible restriction leaves
$(P,Q,u,\ell)$ unchanged. Hence $\mathbf T_{r,\kappa}$ and $\mathbf T_{x,\kappa}$
are principal submatrices of $\mathbf T_r$ and $\mathbf T_x$, and
$z\in\mathcal C$ implies $z_\kappa\in\mathcal C_\kappa$ for every
$\kappa\in\mathcal S(\mathcal G)$. In particular, $\mathcal C_\kappa$ follows
from the path products of C1 on the segments within $\kappa$ alone\rev{, evaluated with the full-network injection bounds}.
\end{rem}

\begin{lemm}
\label{lem:gan-triangular-signs}
If $z\in\mathcal C$, then
\[
 \hat{\mathbf Z}\hat{\mathbf U}^{-1}\in\R^{n\times n}_{-},\;\mathrm{and}\;
 \hat{\mathbf U}^{-\top}\zeta\in\R^n_{+}\quad\mathrm{with}\;\zeta\in\{\mathscr r,\mathscr x\}.
\]
\end{lemm}
\begin{revblock}
\begin{proof}
Since $\mathbf I+2\Gamma=(\mathbf I+\hat C_f)(-\hat C^{-1})$, the definitions give
\[
 \hat{\mathbf Z}\hat{\mathbf U}^{-1}
 =-\mathbf R(\mathbf I+\hat C_f)\mathbf T_r
  -\mathbf X(\mathbf I+\hat C_f)\mathbf T_x\in\R^{n\times n}_{-}.
\]
For $\zeta\in\{\rr,\xx\}$, use $-\hat C^{-\top}c_0=\mathbf1$ to obtain $\hat{\mathbf U}^{-\top}\zeta=\mathbf T_\zeta^\top c_0\in\R^n_{+}$.
\end{proof}
\end{revblock}
\rev{These two signs are all that the factorization $\mathbf M_\ell=\mathbf W\hat{\mathbf U}$ requires: the first makes $\mathbf W$ a $Z$-matrix, and the second will sign the weighted total losses.}

\begin{prop}
\label{prop:gan-transformed-M}
If $z\in\mathcal R_{\mathrm{svs}}\cap \mathcal C$, then
\[\mathbf W =\mathbf I - \mathbf L \hat{C}^{-\top}\hat{\mathbf Z}\hat{\mathbf U}^{-1}\]
is a nonsingular $M$-matrix, and
\[
\mathbf K_u\in\R^{n\times n}_{+},\qquad
\mathbf M_\ell^{-\top}\mathscr r\in\R^{n}_{+},\qquad
\mathbf M_\ell^{-\top}\mathscr x\in\R^{n}_{+}.
\]
Moreover, for $\alpha,\beta\ge0$,
\begin{equation}\label{eq:gan-strict-weighted-inverse}
\xi:=\alpha\mathscr r+\beta\mathscr x\in\R_{++}^n\quad\Longrightarrow\quad\mathbf M_\ell^{-\top}\xi\in\R_{++}^n.
\end{equation}
\end{prop}

\begin{revblock}
\begin{proof}
Lemma~\ref{lem:gan-triangular-signs} makes $\mathbf W$ a $Z$-matrix. \rev{Since $\hat{\mathbf U}^{-1}$ is upper triangular, $\mathbf W^{[k]}=\mathbf M_\ell^{[k]}(\hat{\mathbf U}^{-1})^{[k]}$, and} the positive prefix minors from Lemma~\ref{lemm::positive::Ml} give
\[
 \det\mathbf W^{[k]}=
 \frac{\det\mathbf M_\ell^{[k]}}{\prod_{j=1}^k u_j}>0,
 \qquad k=1,\ldots,n.
\]
Hence $\mathbf W$ is a nonsingular $M$-matrix~\cite[Thm.~6.2.3.E17]{berman1994nonnegative}. Factoring $\mathbf M_\ell=\mathbf W\hat{\mathbf U}$ gives $\mathbf K_u=\hat C^{-\top}\hat{\mathbf Z}\hat{\mathbf U}^{-1}\mathbf W^{-1}\in\R^{n\times n}_{+}$ and
\begin{equation*}
 \mathbf M_\ell^{-\top}\xi
 =\mathbf W^{-\top}\hat{\mathbf U}^{-\top}\xi\in\R^n_{+},
 \quad \xi\in\{\rr,\xx\}.
\end{equation*}
\rev{The strict inequality~\eqref{eq:gan-strict-weighted-inverse} requires a separate argument along ancestor paths, given in Appendix~\ref{app:proof-prop-gan-transformed-M}.}
\end{proof}
\end{revblock}
\rev{The proposition shows why $\mathcal C$ suffices without $\mathbf M_\ell^{-1}\ge0$: the sign information is carried by $\hat{\mathbf U}^{-1}$ through~\eqref{eq:T}, and $\mathbf W^{-1}\ge0$ propagates it.}

\rev{The two regions are not nested in general, but they share a common core.} For comparison, let $\mathcal F_s\subseteq\mathcal F$ impose
$\rr_jP_i+\xx_jQ_i\ge0$ for every $i\in\mathcal N_j^{\anc}$ and $j\in\mathcal N$, without the loss allowance in~\eqref{eq:ancestor-coupling}.
\begin{col}
\label{col:continuation-gan-region}
\[
\operatorname{comp}_{z^0}(\mathcal R_{\mathrm{vs}}\cap\mathcal F_s)
\subseteq\mathcal R_{\mathrm{svs}}\cap\mathcal F_s
\subseteq\mathcal R_{\mathrm{svs}}\cap\mathcal C.
\]
\end{col}
\begin{proof}
The first inclusion follows from Corollary~\ref{col:continuation-strong-region}. On $\mathcal F_s$, $\hat{\mathbf U}$ is an upper triangular $Z$-matrix with positive diagonal, so $\hat{\mathbf U}^{-1}\in\R^{n\times n}_{+}$ and~\eqref{eq:T} holds.
\end{proof}
A similar relationship between C1 from~\cite[Thm.~1]{gan2014exact} and no reverse power flows, i.e., $P,Q\in\R^n_{+}$, can be found in \rev{statement 1) of}~\cite[Thm.~4]{gan2014exact}.

\subsection{Monotonicity with respect to Slack Voltage}
\label{sec:slack-voltage-monotonicity}

\rev{The sign properties of Sections~\ref{sec::monotonicity::ul} and~\ref{sec:gan-matrix-certificate} now translate the derivative formulas~\eqref{eq::passive::dus} into monotone responses to the slack voltage. We distinguish passive systems, where the nodal injections are held fixed, from dispatchable systems, where the receiving-end branch powers can be held fixed instead.}

\subsubsection{Passive Systems}
\label{sec:fixed-injection-slack}
Applying the two certificates, i.e., Propositions~\ref{thm:strong-region-M} \& \ref{prop:gan-transformed-M} to Eqs.~\eqref{eq::dldus}\eqref{eq::dudus}, we have the following network responses to the slack \rev{voltage} $u_{s}$ for passive systems, i.e., when nodal power injections $(p,q)$ are \rev{fixed}:
\begin{col}[Voltage and current monotonicity]
\label{prop:slack-monotonicity-injections}
If $z\in\mathcal R_{\mathrm{svs}}\cap\mathcal F$, then
\begin{align*}
 \partial_{u_s}\ell^\star(p,q;u_s)&\in\R^{n}_{-},\\
 \partial_{u_s}u^\star(p,q;u_s)-\mathbf 1& \in\R^{n}_{+}.
\end{align*}
More precisely,
\begin{equation*}
 [\partial_{u_s}\ell^\star(p,q;u_s)]_j
 \in
 \begin{dcases}
\R_{--}, & \ell_j>0,\\
\{0\}, & \ell_j=0.\\
 \end{dcases}
\end{equation*}
\end{col}

\begin{revblock}
\begin{proof}
The nonnegative inverse in Proposition~\ref{thm:strong-region-M} and $\hat C^{-\top}\hat{\mathbf Z}\in\R^{n\times n}_{+}$ give the weak signs from~\eqref{eq::dldus}--\eqref{eq::dudus}. Set $\rev{\beta}=\mathbf M_\ell^{-1}\ell\in\R^n_{+}$. If $\ell_j>0$, the positive inverse diagonal gives $\rev{\beta}_j>0$. If $\ell_j=0$, exactness gives $P_j=Q_j=0$ and row $j$ of $\mathbf M_\ell$ is $u_j\mathbf e_j^\top$. Hence $u_j\rev{\beta}_j=\ell_j=0$, proving the strictness characterization.
\end{proof}
\end{revblock}
\rev{The observation that $\ell_j=0$ reduces row $j$ of $\mathbf M_\ell$ to $u_j\mathbf e_j^\top$ is reused below.}

\begin{col}[Voltage and total-loss monotonicity]
\label{col:gan-slack-monotonicity}
If $z\in\mathcal R_{\mathrm{svs}}\cap \mathcal C$, then
\begin{align*}
 \partial_{u_s}u^\star(p,q;u_s)-\mathbf 1& \in\R^{n}_{+},
\end{align*}
and the total active and reactive losses satisfy
\begin{align*}
   \partial_{u_s}\big(\mathscr \zeta^\top\ell^\star(p,q;u_s)\big)\leq0,\qquad\mathrm{with}\quad \zeta\in\{\rr,\xx\}.
\end{align*}
More precisely, each loss decreases strictly if and only if it is non-lossless at the operating point:
 \begin{align}\label{eq::strict::ineq}
   \partial_{u_s}\big(\mathscr \zeta^\top\ell^\star(p,q;u_s)\big)<0
 &\;\Longleftrightarrow\; \mathscr \zeta^\top\ell>0\;\mathrm{with}\;  \zeta\in\{\rr,\xx\}.
 \end{align}
\end{col}

\begin{revblock}
\begin{proof}
By Proposition~\ref{prop:gan-transformed-M}, $\mathbf K_u\in\R_+^{n\times n}$. Equation~\eqref{eq::dudus} therefore gives
\[
 \partial_{u_s}u^\star(p,q;u_s)-\mathbf1
 =\mathbf K_u\ell\in\R_+^n.
\]
Fix $\zeta\in\{\mathscr r,\mathscr x\}$ and set $b:=\hat{\mathbf U}^{-\top}\zeta\in\R_+^n$, where nonnegativity follows from Lemma~\ref{lem:gan-triangular-signs}. Since $\mathbf M_\ell=\mathbf W\hat{\mathbf U}$ \rev{and $\mathbf W^{-1}=\hat{\mathbf U}\mathbf M_\ell^{-1}=(\mathbf M_\ell+\mathbf L\hat C^{-\top}\hat{\mathbf Z})\mathbf M_\ell^{-1}=\mathbf I_n+\mathbf L\mathbf K_u$}, Equation~\eqref{eq::dldus} yields
\[
 \partial_{u_s}\bigl(\zeta^\top\ell^\star(p,q;u_s)\bigr)
 =-\zeta^\top\mathbf M_\ell^{-1}\ell
 =-b^\top\mathbf W^{-1}\ell
 \rev{=-b^\top\mathbf L\bigl(\mathbf1+\mathbf K_u\ell\bigr)\le0.}
\]
\rev{The equivalence between strict decrease and $\zeta^\top\ell>0$ is verified in Appendix~\ref{app:proof-col-gan-slack-monotonicity}.}
\end{proof}
\end{revblock}
\rev{The identity $\mathbf W^{-1}=\mathbf I_n+\mathbf L\mathbf K_u$ shows where the total-loss sign comes from: a slack-voltage increase raises every voltage, and under $\mathcal C$ the resulting current changes are weighted by the nonnegative vector $b$. Unlike Corollary~\ref{prop:slack-monotonicity-injections}, the $\mathcal C$ certificate does not sign individual current derivatives. Only voltages and impedance-weighted total losses are signed.}

\subsubsection{Dispatchable Systems}

For a dispatchable subtree, injection adjustments permit a second response in which $(P,Q)$ are fixed. Use the alternative partition
\begin{equation*}
 x_d:=(p,q,u),\qquad y_d:=(\ell,P,Q).
\end{equation*}
The affine Jacobian $\partial_{x_d}g_\lin=\operatorname{diag}(-\mathbf I_n,-\mathbf I_n,-\hat C^\top)$ is nonsingular. Thus the affine equations determine $x_d=x_d^\star(y_d;u_s)$ with
\begin{equation*}
 \mathrm D_{(y_d,u_s)}x_d^\star(y_d,u_s)=
 \left[\begin{array}{ccc;{2pt/2pt}c}
 \hat C_f\mathbf R&\hat C&\mathbf0_n&0\\
 \hat C_f\mathbf X&\mathbf0_n&\hat C&0\\
 \hat C^{-\top}\mathbf Z^2&2\hat C^{-\top}\mathbf R
                         &2\hat C^{-\top}\mathbf X&\mathbf1
 \end{array}\right],
\end{equation*}
with
\begin{equation}\label{eq::partial::xy::receiving}
 \partial_\ell x_d^\star=
 \begin{bmatrix}\hat C_f\mathbf R\\\hat C_f\mathbf X\\\hat C^{-\top}\mathbf Z^2\end{bmatrix},
 \qquad \partial_{u_s}x_d^\star=\begin{bmatrix}0\\0\\\mathbf1\end{bmatrix}.
\end{equation}
Define the reduced nonlinear residual $$\hat g_\non^r(y_d;u_s):=g_\non(x_d^\star(y_d;u_s),y_d;u_s).$$ Its derivative is
\begin{equation*}
 \mathrm D_{(y_d,u_s)}\hat g_\non^r(y_d,u_s)
 =-\begin{bmatrix}\mathbf M_\ell^r&\mathbf M_P^r&\mathbf M_Q^r&\ell\end{bmatrix},
\end{equation*}
with
\begin{align*}
 \mathbf M_\ell^r:=&\;\mathbf U+\mathbf L\hat C^{-\top}\mathbf Z^2,\\
  \mathbf M_P^r:=&\;2(\mathbf L\hat C^{-\top}\mathbf R-\mathbf P),\notag\\
 \mathbf M_Q^r:=&\;2(\mathbf L\hat C^{-\top}\mathbf X-\mathbf Q).\notag
\end{align*}
\rev{The superscript $r$ marks the reduction at fixed receiving-end powers. Unlike $\mathbf M_\ell$, the matrix $\mathbf M_\ell^r$ is triangular, so regularity alone fixes its signs.}
\begin{prop}\label{prop::nonsingular::Mr}
If
$$z\in\mathcal R_{\mathrm{svs}},$$
then $\mathbf{M}^r_\ell$ is a nonsingular $M$-matrix and $$(\mathbf{M}^r_\ell)^{-1}\in\R^{n\times n}_{+}.$$
\end{prop}
\begin{proof}
Since $z\in\mathcal R_{\mathrm{svs}}$, we have $u\in\R^n_{++}$ and $\ell\in\R^n_{+}$, and the path structure (Lemma~\ref{lem:incidence-matrix}) makes $\mathbf M_\ell^r$ a lower triangular $Z$-matrix:
\begin{equation*}
 [\mathbf M_\ell^r]_{ij}=
 \begin{cases}
 u_i-|\zz_i|^2\ell_i,&i=j,\\
 -\ell_i|\zz_j|^2,&j\in\mathcal N_i^{\anc},\\
 0,&\text{otherwise}.
 \end{cases}
\end{equation*}%
\rev{Its leading principal minors are the products $\prod_{j\le k}(u_j-|\zz_j|^2\ell_j)$, which are positive by Lemma~\ref{lem:strong-branch-limit}. Hence $\mathbf M_\ell^r$ is a nonsingular $M$-matrix by~\cite[Thm.~6.2.3.E17]{berman1994nonnegative}.}
\end{proof}
Similar to Eq.~\eqref{eq::passive::dus}, we have
\begin{subequations}\label{eq::dispatch::ul}
  \begin{align}
 \partial_{u_s}\ell^\star(P,Q;u_s)
 &=-(\mathbf M_\ell^r)^{-1}\ell,
 \\
 \partial_{u_s}u^\star(P,Q;u_s)
 &=\mathbf1-\hat C^{-\top}\mathbf Z^2
                  (\mathbf M_\ell^r)^{-1}\ell.
\end{align}
\end{subequations}
\rev{Analogous} to Corollary~\ref{prop:slack-monotonicity-injections}, we have the following statement:
\begin{col}[Monotonicity at fixed receiving-end powers]
\label{prop:slack-monotonicity-receiving}
If $z\in\mathcal R_\mathrm{svs}$, then
\begin{align*}
 \partial_{u_s}\ell^\star(P,Q;u_s)&\in\R^{n}_{-},\\
 \partial_{u_s}u^\star(P,Q;u_s)-\mathbf 1& \in\R^{n}_{+}.
\end{align*}
More precisely,
\begin{equation*}
 [\partial_{u_s}\ell^\star(P, Q;u_s)]_j
 \in
 \begin{dcases}
\R_{--}, & \ell_j>0,\\
\{0\}, & \ell_j=0.\\
 \end{dcases}
\end{equation*}
\end{col}
\begin{proof}
By Proposition~\ref{prop::nonsingular::Mr}, we have $(\mathbf{M}^r_\ell)^{-1}\in\R^{n\times n}_+$.
For $\beta:=(\mathbf M_\ell^r)^{-1}\ell\in\R^n_{+}$, \rev{row $j$ of $\mathbf M_\ell^r\beta=\ell$ reads}
\[
 \rev{(u_j-|\zz_j|^2\ell_j)}\beta_j=\ell_j\left(1+\sum_{k\in\mathcal N_j^{\anc}}|\zz_k|^2\beta_k\right).
\]
Hence $\beta_j>0$ exactly when $\ell_j>0$. The voltage sign follows from $-\hat C^{-\top}\mathbf Z^2\beta\in\R^n_{+}$.
\end{proof}
Following from~\eqref{eq::partial::xy::receiving}, we also have the injection adjustments:
\begin{subequations}\label{eq::dispatch::pq}
\begin{align}
 \partial_{u_s}p^\star(P,Q;u_s)
 &=-\hat C_f\mathbf R(\mathbf M_\ell^r)^{-1}\ell\in\R^n_{-},\\
 \partial_{u_s}q^\star(P,Q;u_s)
 &=-\hat C_f\mathbf X(\mathbf M_\ell^r)^{-1}\ell\in\R^n_{-}.
\end{align}
\end{subequations}
Their compatibility with operating bounds is checked in the hybrid certificate.

\section{Exactness in Different Systems}
\label{sec:exactness-certificates}

\rev{We now turn from the DistFlow equations to the OPF problem~\eqref{eq:opf-distflow-split}. The sensitivities of Section~\ref{sec:regularity-sensitivity} are used to certify positive nonlinear multipliers, first for fully dispatchable and passive systems in this section, and then, in Section~\ref{sec:hybrid-exactness}, for their interconnection. Section~\ref{sec:normal-cones} states the multiplier certificate itself. Sections~\ref{sec:fully-controllable} and~\ref{sec:passive-fixed-injection} construct, for each type of system, a first-order direction that isolates one branch multiplier and determines its sign.}

\subsection{Exactness Certificate via Nonlinear Multipliers}
\label{sec:normal-cones}

To retain active operating bounds in the multiplier argument, we express stationarity using normal and tangent cones. For a closed convex set $\mathcal S$ and $w\in\mathcal S$, write
\begin{align*}
 N_{\mathcal S}(w)
 :=&\{\xi:\langle\xi,\widetilde w-w\rangle\le0
                \text{ for every }\widetilde w\in\mathcal S\},\\
 T_{\mathcal S}(w)
 :=&\operatorname{cl}\operatorname{cone}(\mathcal S-w).
\end{align*}
Set $N_{\mathcal S}(w)=\varnothing$ outside $\mathcal S$. Normal-tangent polarity gives
\begin{equation}\label{eq:normal-tangent-polarity}
 \xi\in N_{\mathcal S}(w),\quad d\in T_{\mathcal S}(w)
 \quad\Longrightarrow\quad \langle\xi,d\rangle\le0,
\end{equation}
see~\cite[Ch.~6]{rockafellar1998variational}.
The original problem~\eqref{eq:opf-distflow-split} and its relaxation~\eqref{eq:opf-soc} share the Lagrangian
\[
\mathcal L(z,\lambda_\lin,\lambda_\non):=f(z)+\lambda_\lin^\top g_\lin(z)+\lambda_\non^\top g_\non(z).
\]
For $w=(z,\lambda_\lin,\lambda_\non)$, define
\[
\mathscr F(w):=\begin{bmatrix}\nabla_z\mathcal L(z,\lambda_\lin,\lambda_\non)\\-g_\lin(z)\\-g_\non(z)\end{bmatrix}.
\]
Their \rev{Karush--Kuhn--Tucker (KKT)} systems can be written as generalized equations
\begin{subequations}\label{eq:GE-KKT}
\begin{align}
0&\in\mathscr F(w)+N_{\mathcal Z\times\mathbb R^{3n}\times\mathbb R^n}(w),\label{eq:GE-equality}\\
0&\in\mathscr F(w)+N_{\mathcal Z\times\mathbb R^{3n}\times\mathbb R_+^n}(w).\label{eq:GE-SOC}
\end{align}
\end{subequations}
\rev{A KKT point of the original problem is feasible for the relaxation and satisfies complementarity automatically, since $g_\non(z)=0$. Hence} nonnegative nonlinear multipliers $\lambda_\non$ transfer an original KKT point to the relaxed KKT system.
The stronger condition
\begin{equation*}
 \lambda_\non^*\in\R_{++}^n
\end{equation*}
forces every relaxed optimizer to be exact. Related multiplier arguments appear in~\cite[Lemmas~3-6]{gan2012exactness}. Stating this condition as \rev{a} general certificate allows the passive and dispatchable subtrees to use the same exactness argument, including when a subtree root voltage is a decision variable.
\begin{thm}[Exactness via nonlinear multipliers]
\label{thm:positive-multiplier-transfer}
Under Assumption~\ref{ass:convex-exactness}, let $w^*=(z^*,\lambda_\lin^*,\lambda_\non^*)$ be a KKT point of the original problem~\eqref{eq:opf-distflow-split}. If $\lambda_\non^*\in\R_{++}^n$, then $w^*$ is also a KKT point of the relaxed problem~\eqref{eq:opf-soc}, $z^*$ is a global minimizer of both problems, and every relaxed optimizer satisfies $g_\non(z)=0$. Moreover, if the slack bus voltage is fixed, then $z^*$ is the unique global minimizer of both~\eqref{eq:opf-distflow-split} and~\eqref{eq:opf-soc}.
\end{thm}
\begin{revblock}
\begin{proof}
Normalize the branch constraints as $h_j(z):=(P_j^2+Q_j^2)/u_j-\ell_j$, which are convex for $u_j>0$, and set $\mu_j^*:=u_j^*[\lambda_\non^*]_j>0$. At the exact point $z^*$, $\nabla g_{\non,j}(z^*)=u_j^*\nabla h_j(z^*)$, so~\eqref{eq:GE-equality} is the stationarity condition of the convex function
$\widetilde{\mathcal L}(z):=f(z)+(\lambda_\lin^*)^\top g_\lin(z)+\sum_j\mu_j^*h_j(z)$ over $\mathcal Z$. Hence $\widetilde{\mathcal L}(z)\ge\widetilde{\mathcal L}(z^*)=f(z^*)$ on $\mathcal Z$, and for every relaxed feasible $z$,
\[
 f(z^*)\le\widetilde{\mathcal L}(z)=f(z)+\sum_j\mu_j^*h_j(z)\le f(z).
\]
This proves global optimality for both problems and the relaxed KKT conditions. At a relaxed optimizer $\bar z$, equality throughout forces $h_j(\bar z)=0$ for every $j$, since $\mu_j^*>0$ and $h_j(\bar z)\le0$. Uniqueness under a fixed slack voltage follows from the hollow feasible set argument in~\cite[Thm.~3]{gan2014exact}.
\end{proof}
\end{revblock}

\rev{Theorem~\ref{thm:positive-multiplier-transfer} reduces exactness to the sign of $n$ scalars. The next statement shows how to read off one multiplier at a time from a suitable first-order direction. It} is consistent with the classical sensitivity interpretation of Lagrange multipliers under right-hand-side perturbations; see  \cite[Prop.~3.2.2]{bertsekas1997nonlinear}\cite[\rev{Cor.}~4.8]{pacaud2025sensitivity}. \rev{At a KKT point of~\eqref{eq:opf-distflow-split} there is a normal $\xi^*\in N_{\mathcal Z}(z^*)$ such that}
\begin{equation}\label{eq::first::order::optimality}
 \nabla f(z^*)+\mathrm Dg_\lin(z^*)^\top\lambda_\lin^*
 +\mathrm Dg_\non(z^*)^\top\lambda_\non^*+\xi^*=0.
\end{equation}
\begin{prop}[Individual multiplier formula]
\label{prop:isolating-direction}
At a KKT point satisfying~\eqref{eq:GE-equality}, suppose $\delta^j$ obeys
\begin{equation}\label{eq:isolating-direction}
 \mathrm Dg_\lin(z^*)\delta^j=0,\qquad
 \mathrm Dg_\non(z^*)\delta^j=\eta\mathbf e_j,\quad\eta\ne0.
\end{equation}
Then
\begin{equation}\label{eq:isolating-multiplier}
 [\lambda_\non^*]_j
 =-\frac{\langle\nabla f(z^*)+\xi^*,\delta^j\rangle}{\eta}.
\end{equation}
If
\begin{equation}\label{eq:isolating-objective-condition}
 \langle\nabla f(z^*),\delta^j\rangle<0,\quad
 \delta^j\in T_{\mathcal Z}(z^*),\quad\eta>0,
\end{equation}
then $[\lambda_\non^*]_j>0$.
\end{prop}

\begin{proof}
Take the inner product of~\eqref{eq::first::order::optimality} with $\delta^j$\rev{. Then}~\eqref{eq:isolating-direction} gives the formula~\eqref{eq:isolating-multiplier}\rev{, and} the normal-tangent polarity~\eqref{eq:normal-tangent-polarity} gives \rev{$\langle\xi^*,\delta^j\rangle\le0$, hence the sign under}~\eqref{eq:isolating-objective-condition}.
\end{proof}
\rev{A direction satisfying~\eqref{eq:isolating-direction} is a multiplier test direction: it preserves the affine DistFlow equations and perturbs exactly one nonlinear residual. It need not be feasible for the relaxation. The two classical certificates differ precisely in how such a direction is built, which we show next.}

\subsection{Fully Dispatchable Systems}
\label{sec:fully-controllable}
A system is fully dispatchable when all nodal injections $(p,q)$ are controllable. Dispatchable injections allow \rev{a test direction} to be supported on one branch and its incident buses. \rev{Inspired} by the voltage-preserving perturbation in~\cite[Thm.~1]{farivar2013branch}, we reduce the branch current and adjust the incident powers while holding voltages fixed. Let $i:=\pi(j)$ and $\mathbf e_{ij}:=\mathbf e_{\pi(j)}+\mathbf e_j$\rev{, with $\mathbf e_0:=0\in\R^n$ because the slack injection is not part of $(p,q)$.}
\begin{lemm}[Voltage-preserving direction]
\label{prop:voltage-preserving-direction}
At a feasible point $z\in\mathcal R_\mathrm{svs}$, define
\begin{align}
\delta^j =& \;(\delta^j_P,\delta^j_Q, \delta^j_u,\delta^j_\ell,\delta^j_p,\delta^j_q)\notag\\
=&\;(\tfrac{\rr_j}{2}\mathbf e_j,\tfrac{\xx_j}{2}\mathbf e_j,0,-\mathbf e_j,-\tfrac{\rr_j}{2}\mathbf e_{ij},-\tfrac{\xx_j}{2}\mathbf e_{ij}).\label{eq:voltage-preserving-direction}
\end{align}
Then
\begin{equation*}
 \mathrm Dg_\lin(z)\delta^j=0,\qquad
 \mathrm Dg_\non(z)\delta^j=\eta_j\mathbf e_j,\quad
\end{equation*}
with $\eta_j>0$.
\end{lemm}
\begin{proof}
Substitution using $\hat C\mathbf e_j=\mathbf e_i-\mathbf e_j$ and $\hat C_f\mathbf e_j=\mathbf e_i$ preserves the affine equations~\eqref{eq:DistFlow-matrix:P}-\eqref{eq:DistFlow-matrix:U}. Only nonlinear residual $j$ changes, with
\begin{align*}
 \eta_j=u_j+\rr_jP_j+\xx_jQ_j=\tfrac12(u_i+u_j-|\zz_j|^2\ell_j).
\end{align*}
Then, $\eta_j>0$ follows from $u_{\pi(j)}>0$ and Lemma~\ref{lem:strong-branch-limit}.
\end{proof}
The associated objective decrease is
\begin{align}\label{eq:fully-controllable-monotone-objective}
 d_j^*:=&-\langle\nabla f(z^*),\delta^j\rangle\notag \\=&\partial_{\ell_j}\varphi^*
   +\tfrac{\rr_j}{2}(\partial_{p_i}\varphi^*+\partial_{p_j}\varphi^*)\notag\\
   &\qquad\qquad+\tfrac{\xx_j}{2}(\partial_{q_i}\varphi^*+\partial_{q_j}\varphi^*).
\end{align}
with  $\varphi^*:=f(z^*)=\varphi(u^*,\ell^*,p^*,q^*;p_0(z^*),q_0(z^*))$. \rev{for $i=0$, the derivatives $\partial_{p_i}\varphi^*,\partial_{q_i}\varphi^*$ refer to the slack-injection arguments $p_0,q_0$. Since $\delta_u^j=0$, no voltage-cost sign is needed.}

\begin{thm}[Exactness in dispatchable systems]
\label{thm:fully-controllable-exactness}
Under Assumption~\ref{ass:convex-exactness}, let $w^*=(z^*,\lambda_\lin^*,\lambda_\non^*)$ be a KKT point with fixed slack voltage with $z^*\in\mathcal R_{\mathrm{svs}}$. If, for every $j\in\mathcal N$,
\begin{equation*}
 \delta^j\in T_{\mathcal Z}(z^*),\qquad d_j^*>0,
\end{equation*}
then
\begin{equation*}
 [\lambda_\non^*]_j\ge \frac{d_j^*}{\eta_j^*}>0,
\end{equation*}
and the conclusions of Theorem~\ref{thm:positive-multiplier-transfer} hold.
\end{thm}
\begin{proof}
Proposition~\ref{prop:isolating-direction}, Lemma~\ref{prop:voltage-preserving-direction},  and the normal-tangent polarity~\eqref{eq:normal-tangent-polarity} give
$$[\lambda_\non^*]_j=\frac{d_j^*-\langle\xi^*,\delta^j\rangle}{\eta_j^*}\ge \frac{d_j^*}{\eta_j^*}.$$
\end{proof}
\begin{rem}[Sufficient objective conditions]
\label{rem:fully-controllable-objective}
Nonnegative current and incident-injection derivatives make every term in~\eqref{eq:fully-controllable-monotone-objective} nonnegative. Strictness holds if the current derivative is positive or an endpoint injection derivative has a positive coefficient: $\rr_j>0$ for active power and $\xx_j>0$ for reactive power.
\end{rem}
\begin{rem}[Compatibility with operating bounds]
\label{rem:fully-controllable-bounds}
For box~\eqref{eq:operating-set}, tangent membership requires
\[
 \begin{cases}
 z_k^*>[z_{\text{\normalfont\scshape min}}]_k,&[\delta^j]_k<0,\\
 z_k^*<[z_{\text{\normalfont\scshape max}}]_k,&[\delta^j]_k>0.
 \end{cases}
\]
Hence the lower current bound should be inactive. For $\rr_j>0$, the upper $P_j$ bound and incident non-slack lower $p$ bounds must also be inactive. $\xx_j>0$ gives the analogous reactive requirements. Active voltage bounds cause no obstruction.
\end{rem}
\subsection{Passive Systems}
\label{sec:passive-fixed-injection}
\rev{When no injection can move, the voltage-preserving direction is unavailable, because it requires injection headroom at both ends of the branch. The test direction must instead keep all injections fixed and let the voltages respond, which is exactly the fixed-injection sensitivity of Section~\ref{sec::condensed}.} In a passive system, all non-slack injections $(p,q)=(p^{\mathrm{fix}},q^{\mathrm{fix}})$ are fixed. With slack voltage fixed, regularity isolates a feasible solution and supplies a zero-normal multiplier, even at active operating bounds.

Set $x=(P,Q,u)$ and
\[
 \mathcal Z^{\mathrm{fix}}
 :=\{(x,\ell):(x,\ell,p^{\mathrm{fix}},q^{\mathrm{fix}})\in\mathcal Z\}.
\]
Let $z\rev{^*} = (x^*,\ell^*; p^{\mathrm{fix}},q^{\mathrm{fix}})$ be a feasible point \rev{of} the DistFlow equations~\eqref{eq:DistFlow-matrix},
\begin{equation}\label{eq:passive-reduced-jacobian}
 \mathbf M_\ell^*=-\partial_\ell\hat g_\non(\ell^*,p^{\mathrm{fix}},q^{\mathrm{fix}};u_s)\end{equation}
be the reduced Jacobian matrix, and
\begin{equation}\label{eq:dxdl}
\begin{aligned}
 c_\ell^*\hspace{-1pt}:=\hspace{-1pt}\nabla_\ell\varphi^*\hspace{-2pt}-\hspace{-1pt}(\hat C^{-\top}\hat{\mathbf Z})^\top\nabla_u\varphi^*\hspace{-1pt}+\hspace{-1pt}\partial_{p_0}\varphi^*\rr\hspace{-1pt}+\hspace{-1pt}\partial_{q_0}\varphi^*\xx,
\end{aligned}
\end{equation}
be the effective reduced objective gradient, both \rev{evaluated} at the point $z^*$. \rev{By~\eqref{eq:affine-elimination-slack}, $c_\ell^*=\nabla_\ell f(z^*)+(\partial_\ell x^\star)^\top\nabla_xf(z^*)$ is the gradient of the objective along the fixed-injection solution map.}
\begin{prop}[Multiplier for passive systems]
\label{prop:passive-zero-normal}
Let $z^*\in\mathcal R$ be feasible with fixed nodal injections and slack voltage. There is a KKT point with zero normal in $N_{\mathcal Z^{\mathrm{fix}}}(x^*,\ell^*)$ and
\begin{equation}\label{eq:passive-multiplier-identity}
 \lambda_\non^*=(\mathbf M_\ell^*)^{-\top}c_\ell^*.
\end{equation}
This remains valid at active bounds.
\end{prop}
\begin{proof}
The square Jacobian $J(z^*)$ is nonsingular, so the inverse function theorem \rev{implies} that $z^*$ is \rev{an} isolated feasible solution~\cite[Thm.~1A.1]{dontchev2009implicit}, and thus a local minimizer. For the zero normal, solve
\begin{equation*}
 J(z^*)^\top
 \begin{bmatrix}\lambda_\lin^*\\\lambda_\non^*\end{bmatrix}
 =-\nabla_{(x,\ell)}f(z^*).
\end{equation*}
\rev{Eliminating $\lambda_\lin^*$ through the affine block, as in Section~\ref{sec::condensed},}
yields $c_\ell^*-\mathbf M_\ell^{\rev{*}\top}\lambda_\non^*=0$.
If fixed injections remain variables, normals to their singleton sets absorb their stationarity components. Zero belongs to every operating normal cone, including at active bounds.
\end{proof}
\begin{thm}[Exactness certificate in passive systems]
\label{thm:passive-exactness}
Under Assumption~\ref{ass:convex-exactness}, let $z^*$ be feasible with fixed nodal injections and slack voltage. Suppose either
\begin{enumerate}[label=\arabic*)]
\item $z^*\in\mathcal R_{\mathrm{svs}}\cap\mathcal F$ and $c_\ell^*\in\R_{++}^n$;
\item $z^*\in\mathcal R_{\mathrm{svs}}\cap\mathcal C$ and $c_\ell^*=\alpha\rr+\beta\xx\in\R_{++}^n$, with $\alpha,\beta\ge0$.
\end{enumerate}
Then $\lambda_\non^*$ in~\eqref{eq:passive-multiplier-identity} is strictly positive, and Theorem~\ref{thm:positive-multiplier-transfer} applies.
\end{thm}
\begin{proof}
Proposition~\ref{prop:passive-zero-normal} supplies \rev{the identity}~\eqref{eq:passive-multiplier-identity}. In case 1), Proposition~\ref{thm:strong-region-M} makes $(\mathbf M_\ell^*)^{-1}$ nonnegative with positive diagonal. In case 2), use the strict weighted-inverse property~\eqref{eq:gan-strict-weighted-inverse}. Both give $(\mathbf M^*_\ell)^{-\top}c_\ell^*\in\R^n_{++}$.
\end{proof}
This fixed-injection certificate requires no tangent direction or interior operating point. \rev{When some injections or the slack voltage become decision variables, the zero-normal selection is no longer available, but the same quantity still bounds the multipliers from below along the corresponding test direction.}
\begin{col}[Multiplier lower bound]
\label{col:passive-multiplier-bound}
Under Assumption~\ref{ass:convex-exactness}, let $w^*=(z^*,\lambda_\lin^*,\lambda_\non^*)$ be a KKT point \rev{of~\eqref{eq:opf-distflow-split}} with $z^*\in\mathcal R_{\mathrm{svs}}$, \rev{where some injections $(p,q)$ or the slack voltage $u_s$ may be decision variables}. For \rev{each} $j\in\mathcal N$, define the direction
\begin{align}
\delta^j =& \;(\delta^j_P,\delta^j_Q,\delta^j_u,\delta^j_\ell,\delta^j_p,\delta^j_q)\notag\\
=&\;(\Gamma\mathbf R \delta_\ell^j,\, \Gamma\mathbf X\delta_\ell^j,\, \mathbf K_u^*\mathbf e_j,\,\delta_\ell^j,0,0)\rev{,\qquad \delta^j_{u_s}=0,}
\label{eq::direction::passive}
\end{align}
with
\begin{align*}
 \delta_\ell^j=-(\mathbf M_\ell^*)^{-1}\mathbf e_j\rev{,\qquad \mathbf K_u^*:=\hat C^{-\top}\hat{\mathbf Z}(\mathbf M_\ell^*)^{-1}}.
\end{align*}
Then \rev{$\mathrm Dg_\lin\,\delta^j=0$ and $\mathrm Dg_\non\,\delta^j=\mathbf e_j$. If $\delta^j\in T_{\mathcal Z}(z^*)$, then}
\begin{equation*}
 \rev{[\lambda_\non^*]_j}\ge\rev{[}(\mathbf M_\ell^*)^{-\top}c_\ell^*\rev{]_j}.
\end{equation*}
\rev{If, in addition, either condition of Theorem~\ref{thm:passive-exactness} holds and $\delta^j\in T_{\mathcal Z}(z^*)$ for every $j\in\mathcal N$}, then
\[
 \lambda_\non^*\in\R^n_{++},
\]
and Theorem~\ref{thm:positive-multiplier-transfer} applies\rev{, with uniqueness only for fixed slack voltage. Under either condition, the direction satisfies}
\[
\delta_u^j\in\R^n_{+}, \quad \rr^\top \delta_\ell^j\rev{\le0},\quad\xx^\top \delta_\ell^j\rev{\le0}, \quad\mathrm{and} \quad
\delta_{p_0}^j,\delta_{q_0}^j\leq 0\rev{,}
\]
\rev{where $\delta_{p_0}^j=\rr^\top\delta_\ell^j$ and $\delta_{q_0}^j=\xx^\top\delta_\ell^j$ are the induced slack-injection changes.}
\end{col}
\rev{The direction is $\partial_\ell x^\star\,\delta_\ell^j$ completed by $\delta_\ell^j$, hence tangent to the fixed-injection solution map, and~\eqref{eq:isolating-direction} holds with $\eta=1$. Since $f$ does not depend on $(P,Q)$ directly, $-\langle\nabla f(z^*),\delta^j\rangle=[(\mathbf M_\ell^*)^{-\top}c_\ell^*]_j$, and Proposition~\ref{prop:isolating-direction} gives the bound. The sign verifications are in Appendix~\ref{app:proof-passive-multiplier-bound}.}
\begin{rem}[Sufficient objective conditions]
\label{rem:passive-objective-conditions}
Since $\hat C^{-\top}\hat{\mathbf Z}\in\R^{n\times n}_+$, case 1) holds if
\begin{align*}
 \nabla_u\varphi^*&\in\R^n_-,\\
 \nabla_\ell\varphi^*
   +\partial_{p_0}\varphi^*\,\rr+\partial_{q_0}\varphi^*\,\xx&\in\R^n_{++}.
\end{align*}
For case 2), weighted loss and slack costs
$\varphi=a\rr^\top\ell+b\xx^\top\ell+\psi(p_0,q_0)$,
where $a,b\ge0$ and $\psi$ is nondecreasing in both arguments, give
\[
 c_\ell^*=(a+\partial_{p_0}\psi^*)\rr
             +(b+\partial_{q_0}\psi^*)\xx.
\]
Its componentwise positivity suffices. For a slack cost $f_0(p_0)$, the gradient is $c_\ell^*=f_0'(p_0^*)\rr$. It is positive if $f_0'(p_0^*)>0$ and $\rr\in\R^n_{++}$. Unlike a finite-perturbation argument using strict increase~\cite{gan2014exact}, this first-order certificate requires a positive derivative: $t^3$ is convex and strictly increasing on $t\ge0$, yet has zero derivative at $0$. Under $\mathcal C$, the weighted form is essential. Arbitrary positive current-cost vectors are not covered by its weighted inverse signs.
\end{rem}
\rev{The two directions of this section have different footprints. The voltage-preserving direction~\eqref{eq:voltage-preserving-direction} is supported on one branch and leaves every voltage unchanged. The passive direction~\eqref{eq::direction::passive} raises voltages throughout the network and, when it is used inside one region of a larger network, propagates to every region downstream. Composing the two therefore requires the subtree responses of Section~\ref{sec:slack-voltage-monotonicity}.}

\section{Hybrid Exactness by Graph Decomposition}
\label{sec:hybrid-exactness}

\rev{In this section, we compose the local certificates of Section~\ref{sec:exactness-certificates} into a certificate for a partially dispatchable system, by constructing for every branch a test direction that satisfies~\eqref{eq:isolating-direction} on the full network. Compared with the voltage-preserving direction~\eqref{eq:voltage-preserving-direction} of a dispatchable subtree, the passive direction~\eqref{eq::direction::passive} changes voltages and therefore propagates to the descendants of its subtree. Section~\ref{sec:hybrid-partition} partitions the network and fixes the nomenclature, Section~\ref{sec:hybrid-extension} extends a passive direction through the affected descendants, Section~\ref{sec:hybrid-certificate} states the hybrid certificate, and Section~\ref{sec:der-placement} draws conclusions for the placement of controllable DERs.}

\begin{figure*}[tb]
\centering
\begin{tikzpicture}[x=1cm,y=1cm,>=stealth,
  control/.style={circle,draw=black,fill=black,
    inner sep=0pt,minimum size=5pt},
  fixed/.style={circle,draw=black,fill=white,inner sep=0pt,minimum size=5pt},
  passive/.style={draw=black!65,dashed,line width=.8pt},
  dispatch/.style={draw=black,line width=1.1pt},
  lab/.style={font=\footnotesize,align=center}]

\node[lab] at (0,3.0) {(a) General hybrid partition};
\node[control,label={[font=\footnotesize]above:$s_\kappa$}] (a) at (0,2.2) {};
\node[fixed] (ab) at (-.675,1.525) {};
\node[fixed] (ac) at (.675,1.525) {};
\node[control,label={[font=\footnotesize]left:$s_\sigma$}] (b) at (-1.35,.85) {};
\node[control,label={[font=\footnotesize]right:$s_\gamma$}] (c) at (1.35,.85) {};
\node[control] (d) at (-1.35,-.65) {};
\node[fixed] (e) at (1.35,-.65) {};
\draw[passive,->] (a)--(ab);
\draw[passive,->] (ab)--(b);
\draw[passive,->] (a)--(ac);
\draw[passive,->] (ac)--(c);
\draw[dispatch,->] (b)--(d);
\draw[passive,->] (c)--(e);
\node[lab,fill=white,inner sep=2pt] at (0,.9) {intermediate\\passive $\kappa$};
\node[lab,anchor=east] at (-1.55,.1) {dispatchable\\$\sigma$};
\node[lab,anchor=west] at (1.55,.1) {terminal passive\\$\gamma$};
\node[lab] at (0,-1.2) {$\delta u_{s_\kappa}=0$,\quad
  $\delta u_{s_\sigma},\delta u_{s_\gamma}\ge0$};

\begin{scope}[xshift=7.1cm]
\node[lab] at (0,3.0) {(b) Backbone with terminal subtrees};
\node[control,label={[font=\footnotesize]above:$0$}] (s) at (-2.15,1.8) {};
\node[control,label={[font=\footnotesize]above:$s_\kappa$}] (i) at (-.4,1.8) {};
\node[control,label={[font=\footnotesize]above:$s_\sigma$}] (j) at (1.5,1.8) {};
\node[fixed] (i1) at (-.4,.5) {};
\node[fixed] (i2) at (-.4,-.65) {};
\node[fixed] (j1) at (1.5,.5) {};
\node[fixed] (j2) at (1.5,-.65) {};
\draw[dispatch,->] (s)--(i);
\draw[dispatch,->] (i)--(j);
\draw[passive,->] (i)--(i1);
\draw[passive,->] (i1)--(i2);
\draw[passive,->] (j)--(j1);
\draw[passive,->] (j1)--(j2);
\node[lab] at (-.3,2.4) {dispatchable backbone};
\node[lab,anchor=east] at (-.65,.1) {passive\\$\kappa$};
\node[lab,anchor=west] at (1.75,.1) {passive\\$\sigma$};
\node[lab] at (0,-1.2) {$\delta u_{s_\kappa}=\delta u_{s_\sigma}=0$};
\end{scope}
\end{tikzpicture}
\caption{Subtree partitions. Filled buses have controllable injections; open buses have fixed injections. Solid branches are dispatchable and dashed branches passive. (a) An intermediate passive source changes its child-root voltages. Descendant responses and interface injections complete the direction. (b) A dispatchable backbone keeps all attachment voltages fixed, and terminal passive directions remain local.}
\label{fig:hybrid-partitions}
\end{figure*}
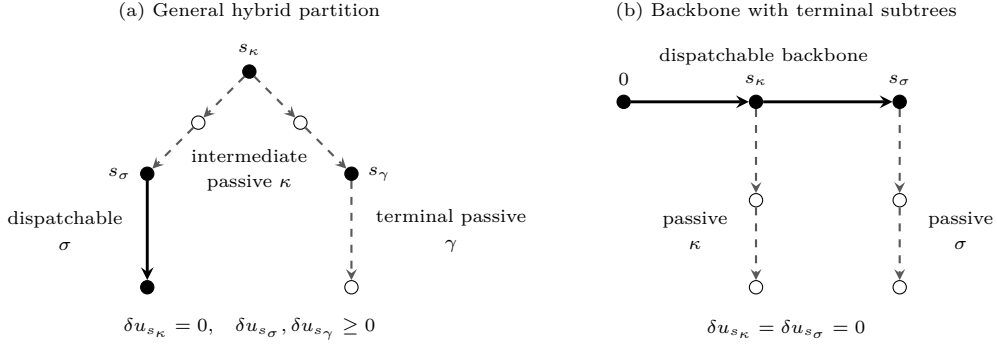

\subsection{Partition into Passive and Dispatchable Subtrees}
\label{sec:hybrid-partition}

The following procedure constructs the partition\rev{; see Fig.~\ref{fig:hybrid-partitions}}.
\begin{enumerate}[label=(\roman*)]
\item Label a branch \emph{dispatchable} if both incident buses have controllable injections. Group connected dispatchable branches into dispatchable subtrees.
\item Group the remaining branches into \emph{passive subtrees}. Starting at a controllable bus, follow these branches downstream through buses with fixed injections, stopping at the next controllable bus or a leaf. Include the stopping bus in the subtree. Branches that meet at a fixed-injection bus stay in the same passive subtree.
\item Call a passive subtree \emph{terminal} if no subtree is attached downstream, and \emph{intermediate} otherwise.
\end{enumerate}
\rev{By construction, the root of every subtree and every bus at which a subtree is attached have controllable injections, while fixed-injection buses are interior to passive subtrees.}

Using the subtree indices from Section~\ref{sec::graph}, write the partition as
\begin{equation*}
\mathcal K\subseteq\mathcal{S}(\mathcal G),\quad \mathcal N=\bigsqcup_{\kappa\in\mathcal K}\mathcal N_\kappa,\quad \mathcal L=\bigsqcup_{\kappa\in\mathcal K}\mathcal L_\kappa.
\end{equation*}
Each non-slack bus with the corresponding branch belongs to exactly one subtree.
We write
\[
\mathcal K=\mathcal K_{\mathrm d}\sqcup\mathcal K_{\mathrm p},\qquad \mathcal K_{\mathrm p}=\mathcal K_{\mathrm{ip}}\sqcup\mathcal K_{\mathrm{tp}},
\]
where the subscripts $\mathrm d$, $\mathrm{ip}$, and $\mathrm{tp}$ denote dispatchable, intermediate passive, and terminal passive subtrees, respectively.
We denote $\mathcal K^{\mathrm{des}}(\kappa)$  as  the strict descendants of subtree $\kappa$. Then,
by definition, we have $\mathcal K^{\mathrm{des}}(\kappa)=\emptyset$ for $\kappa\in\mathcal K_{\mathrm{tp}}$.

A subtree $\sigma\in\mathcal K^{\mathrm{des}}(\kappa)$ is a child of $\kappa$ if its slack bus satisfies $s_\sigma\in\mathcal N_\kappa$. Eq.~\eqref{eq:compatible-subtree-injections} ensures consistency between the full-network operating point $z$ and the compatible subtree operating points $z_\kappa$ and $z_\sigma$. Following~\eqref{eq::slack::balance}, define the local slack injections of subtree $\sigma$ by
\begin{subequations}\label{eq::slack::balance::local}
  \begin{align}
    p_0^\sigma&=\sum_{\substack{j:\pi(j)=s_\sigma\\
    j\in\mathcal N_\sigma}}(P_j+\rr_j\ell_j),\\
    q_0^\sigma&=\sum_{\substack{j:\pi(j)=s_\sigma\\j\in\mathcal N_\sigma}}(Q_j+\xx_j\ell_j). \end{align}
  \end{subequations}
    Similar to~\eqref{eq::slack::injections}, these satisfy \rev{the} power balances of the subtree~$\sigma$
    \begin{equation}\label{eq::slack::injections::local} p_0^\sigma+\mathbf1^\top p^\sigma=\rr_\sigma^\top\ell_\sigma,\qquad q_0^\sigma+\mathbf1^\top q^\sigma=\xx_\sigma^\top\ell_\sigma.
    \end{equation}
    For each bus $i\in\mathcal N_\kappa$, the full-network injections satisfy
    \begin{subequations}\label{eq::slack::coupling}
      \begin{align}
        p_i&=p_i^\kappa+\sum_{s_\sigma=i}p_0^\sigma,\\ q_i&=q_i^\kappa+\sum_{s_\sigma=i}q_0^\sigma,
      \end{align}
    \end{subequations}
    where each sum runs over children $\sigma$ of $\kappa$ rooted at $i$. These identities specify how controllable interface injections must adjust to changes in child-subtree input powers \rev{so that the compatible injections of the parent subtree remain unchanged}.

\subsection{Local Directions to the \rev{Full} Network}
\label{sec:hybrid-extension}

To apply Proposition~\ref{prop:isolating-direction} to the \rev{full} network, we extend the local subtree directions~\eqref{eq:voltage-preserving-direction} and~\eqref{eq::direction::passive} to the full network.

Let $z\in\mathcal R_{\mathrm{svs}}$ be feasible for the original problem~\eqref{eq:opf-distflow-split}. For each subtree $\kappa\in\mathcal K$, all local matrices \rev{$\mathbf M_{\ell,\kappa}$, $\mathbf K_{u,\kappa}$, and the regions $\mathcal F_\kappa$, $\mathcal C_\kappa$ defined by~\eqref{eq::F} and~\eqref{eq::C}} are evaluated at the compatible operating point $(z_\kappa,u_{s_\kappa})$ from Definition~\ref{def:compatible-subtree-operating-point}. Lemma~\ref{lem:strong-subtree-inheritance} ensures that $z_\kappa\in\mathcal R_{\mathrm{svs},\kappa}$.

For a dispatchable subtree $\kappa\in\mathcal K_{\mathrm d}$, the direction~\eqref{eq:voltage-preserving-direction} changes only branch $j\in\mathcal N_\kappa$ and its incident injections. All voltages remain fixed, so no descendant response is needed.

For a passive subtree $\kappa\in\mathcal K_{\mathrm p}$, regularity ensures that the local matrix $\mathbf M_{\ell\rev{,\kappa}}$ is nonsingular. For each $j\in\mathcal N_\kappa$, set $\delta_{u_{s_\kappa}}^j=0$ and apply~\eqref{eq::direction::passive} locally to obtain
\begin{equation}\label{eq:hybrid-passive-source}
\delta_{u_\kappa}^j=\mathbf K_{u,\kappa}\mathbf e_{j,\kappa},
\end{equation}
where $\mathbf e_{j,\kappa}\in\R^{n_\kappa}$ is the local canonical basis vector and $\mathbf K_{u,\kappa}$ is the local sensitivity matrix defined as in~\eqref{eq::Ku}. If $z_\kappa\in\mathcal F_\kappa\cup\mathcal C_\kappa$, Propositions~\ref{thm:strong-region-M} and~\ref{prop:gan-transformed-M} imply
\[
\delta_{u_\kappa}^j\in\R_+^{n_\kappa}.
\]
Hence, within \rev{an} intermediate passive subtree, an increase in an interface voltage also raises the root voltage of each child attached there. The local direction must therefore be extended through the affected descendant subtrees. The following lemma provides responses to the increased root voltage that preserve the local DistFlow equations~\eqref{eq:DistFlow-matrix} to first order.
\begin{lemm}[Responses to \rev{a root-voltage change}]\label{lem:hybrid-descendant-response} Let $\sigma\in\mathcal K$, suppose $z_\sigma\in\mathcal R_{\mathrm{svs},\sigma}$, and let $\tau_\sigma:=\delta_{u_{s_\sigma}}\ge0$ denote the root-voltage change. In each of the following cases, there exists a local response direction $\delta_\sigma$ at $z_\sigma$ that preserves the local DistFlow equations~\eqref{eq:DistFlow-matrix} to first order:
  \begin{enumerate}[label=(\roman*)]
    \item If $\sigma\in\mathcal K_{\mathrm p}$ and $z_\sigma\in\mathcal F_\sigma$, hold $(p^\sigma,q^\sigma)$ fixed. Then, \[ \delta_{u_\sigma}-\tau_\sigma\mathbf1\in\R_+^{n_\sigma},\qquad \delta_{\ell_\sigma}\in\R_-^{n_\sigma}. \]
     \item If $\sigma\in\mathcal K_{\mathrm p}$ and $z_\sigma\in\mathcal C_\sigma$, hold $(p^\sigma,q^\sigma)$ fixed. Then, \[ \delta_{u_\sigma}-\tau_\sigma\mathbf1\in\R_+^{n_\sigma},\qquad \rr_\sigma^\top\delta_{\ell_\sigma}\le0,\qquad \xx_\sigma^\top\delta_{\ell_\sigma}\le0. \]
     \item If $\sigma\in\mathcal K_{\mathrm d}$, hold $(P_\sigma,Q_\sigma)$ fixed. Then,
      \[
      \delta_{u_\sigma}-\tau_\sigma\mathbf1\in\R_+^{n_\sigma},
      \quad
      \delta_{\ell_\sigma}\in\R_-^{n_\sigma},
      \quad
      \delta_{p^\sigma},\delta_{q^\sigma}\in\R_-^{n_\sigma}.
      \]
    \end{enumerate} Each response also satisfies $\delta_{p_{0}^{\sigma}},\delta_{q_{0}^{\sigma}}\le0$.
    \end{lemm}

\rev{Each response is $\tau_\sigma$ times the corresponding slack-voltage sensitivity of Section~\ref{sec:slack-voltage-monotonicity}, so the local residual derivatives vanish by construction, and the signs are those of Corollaries~\ref{prop:slack-monotonicity-injections}, \ref{col:gan-slack-monotonicity}, and~\ref{prop:slack-monotonicity-receiving} together with~\eqref{eq::dispatch::pq}. The slack-injection signs follow by differentiating the local balances. The verification is in Appendix~\ref{app:proof-hybrid-lemmas}.}

\begin{revblock}
The completed direction for a passive branch $j\in\mathcal N_\kappa$ is assembled in three steps. First, apply~\eqref{eq::direction::passive} on the source subtree $\kappa$ with $\delta^j_{u_{s_\kappa}}=0$ and $\delta^j_{p^\kappa}=\delta^j_{q^\kappa}=0$, which gives $\delta^j_{\ell_\kappa}=-\mathbf M_{\ell,\kappa}^{-1}\mathbf e_{j,\kappa}$ and~\eqref{eq:hybrid-passive-source}. Second, proceed through the descendants of $\kappa$ in parent-to-child order: for each $\sigma$, set $\tau_\sigma^j$ equal to the voltage change already assigned at $s_\sigma$ by its parent, and take the response of Lemma~\ref{lem:hybrid-descendant-response}; all other subtrees are left unchanged. Third, recover the physical injections by differentiating~\eqref{eq::slack::coupling}: for $i\in\mathcal N_\gamma$,
\begin{subequations}\label{eq::adjustments}
\begin{align}
\delta_{p_i}^j
&=\delta_{p_i^\gamma}^j
+\sum_{s_\sigma=i}\delta_{p_0^\sigma}^j,\\
\delta_{q_i}^j
&=\delta_{q_i^\gamma}^j
+\sum_{s_\sigma=i}\delta_{q_0^\sigma}^j,
\end{align}
\end{subequations}
where the sums run over the children of $\gamma$ rooted at $i$. Only subtree roots and attachment buses, all of which are controllable, receive nonzero adjustments. The following lemma records the properties of this direction.
\end{revblock}
\begin{lemm}[Extension to the \rev{full} network]
\label{lem:hybrid-global-extension}
Let $z\in\mathcal R_{\mathrm{svs}}$ be feasible for~\eqref{eq:opf-distflow-split}, and let $\kappa\in\mathcal K_{\mathrm p}$. Suppose
\[
z_\sigma\in\mathcal F_\sigma\cup\mathcal C_\sigma,
\qquad
\forall\sigma\in\{\kappa\}\cup
\bigl(\mathcal K^{\mathrm{des}}(\kappa)\cap\mathcal K_{\mathrm p}\bigr),
\]
where the certificate regions~\eqref{eq::F} and~\eqref{eq::C} are evaluated at the compatible subtree operating points. Then, for every $j\in\mathcal N_\kappa$, there exists a direction $\delta^j$ that keeps the slack voltage and all fixed nodal injections unchanged and satisfies
\begin{equation}\label{eq:hybrid-global-isolation}
\mathrm Dg_\lin(z)\delta^j=0,
\qquad
\mathrm Dg_\non(z)\delta^j=\mathbf e_j,
\end{equation}
with
\[
\begin{gathered}
\delta_u^j\in\R_+^n,
\qquad
\delta_p^j,\delta_q^j\in\R_-^n,\\
\delta_{p_0}^j,\delta_{q_0}^j\le0,
\qquad
\rr^\top\delta_\ell^j,\xx^\top\delta_\ell^j\le0,
\end{gathered}
\]
where $\delta_{p_0}^j$ and $\delta_{q_0}^j$ are the induced changes in the physical slack injections. If $\rr_\kappa\in\R_{++}^{n_\kappa}$, then
\begin{equation}\label{eq:hybrid-strict-source-loss}
\delta_{p_{s_\kappa}}^j
=\rr_\kappa^\top\delta_{\ell_\kappa}^j<0,
\quad\mathrm{and}\quad
\rr^\top\delta_\ell^j<0.
\end{equation}
\end{lemm}
\rev{The direction preserves all affine equations because shared root voltages are matched and~\eqref{eq::adjustments} restores the nodal balances. Only the source contributes a nonlinear residual, $\mathbf e_{j,\kappa}$. The signs follow from Lemma~\ref{lem:hybrid-descendant-response} by induction along the descendants, and the strict inequality~\eqref{eq:hybrid-strict-source-loss} from the positive weighted inverse of $\mathbf M_{\ell,\kappa}$ under $\mathcal F_\kappa$ or $\mathcal C_\kappa$. The verification is in Appendix~\ref{app:proof-hybrid-lemmas}.}

\subsection{Hybrid exactness}
\label{sec:hybrid-certificate}

With Proposition~\ref{prop:isolating-direction}, we now combine the dispatchable and passive directions \rev{into} an exactness certificate for the full network. \rev{The tangent condition ensures compatibility with the operating limits, while the objective conditions ensure a strict decrease along each completed direction.}

\begin{thm}[Hybrid exactness]
\label{thm:hybrid-exactness}
Consider a primal-dual KKT point $(z^*,\lambda_\lin^*,\lambda_\non^*)$ of the original problem~\eqref{eq:opf-distflow-split}, with $z^*\in\mathcal R_{\mathrm{svs}}$ and the subtree partition $\mathcal K$ from Section~\ref{sec:hybrid-partition}. Assume $\rr\in\R^n_{++}$, and let every passive subtree satisfy
\[
z_\kappa^*\in\mathcal F_\kappa\cup\mathcal C_\kappa,
\qquad \forall\kappa\in\mathcal K_{\mathrm p}.
\]
Let $\delta^j$ be the voltage-preserving direction~\eqref{eq:voltage-preserving-direction} on a dispatchable branch, or the completed direction of Lemma~\ref{lem:hybrid-global-extension} on a passive branch. Suppose
\begin{equation}\label{eq:hybrid-tangent-condition}
\delta^j\in T_{\mathcal Z}(z^*),\qquad \forall j\in\mathcal N.
\end{equation}
For the objective $f(z)=\varphi(u,\ell,p,q;p_0(z),q_0(z))$ under Assumption~\ref{ass:convex-exactness}, evaluate all derivatives below at $z^*$ and assume
\begin{equation}\label{eq:hybrid-objective-condition}
\hspace{-6pt}\nabla_u\varphi^*\hspace{-4pt}\in\hspace{-2pt}\R_-^n,\qquad
\nabla_{(p_0,p)}\varphi^*,\nabla_{(q_0,q)}\varphi^*
\hspace{-2pt}\in\hspace{-4pt}\R_+^{n+1}.
\end{equation}
and $$\nabla_\ell\varphi^*=\alpha\rr,\qquad\mathrm{with}\;\alpha\ge0.$$ If either of the following conditions holds:
\begin{enumerate}[label=(\roman*)]
  \item $\alpha>0$, or
  \item  $\partial_{p_i}\varphi^*>0$ at every controllable bus $i$, including the slack bus,
\end{enumerate}
then $\lambda_\non^*\in\R_{++}^n$, and the conclusions of Theorem~\ref{thm:positive-multiplier-transfer} hold\rev{, with uniqueness when the slack voltage is fixed}.
\end{thm}
\begin{proof}
For a dispatchable branch $j$, Lemma~\ref{prop:voltage-preserving-direction} gives~\eqref{eq:isolating-direction} with $\eta_j>0$ when $z^*\in\mathcal R_{\mathrm{svs}}$. Under either condition (i) or (ii), the objective decrease $d_j^*$ in~\eqref{eq:fully-controllable-monotone-objective} is strictly positive\rev{: in (i), $\partial_{\ell_j}\varphi^*=\alpha\rr_j>0$; in (ii), both incident buses are controllable and $\rr_j>0$.}

For a passive branch $j\in\mathcal N_\kappa$, Lemma~\ref{lem:hybrid-global-extension} gives~\eqref{eq:isolating-direction} with $\eta_j=1$ and
\[
\delta_u^j\ge0,\qquad
\delta_p^j,\delta_q^j\le0,\qquad
\delta_{p_0}^j,\delta_{q_0}^j\le0.
\]
By~\eqref{eq:hybrid-objective-condition}, the voltage and injection terms in $\langle\nabla f(z^*),\delta^j\rangle$ are therefore nonpositive.
Under (i), the current term is strictly negative by~\eqref{eq:hybrid-strict-source-loss}:
\[
\langle\nabla_\ell\varphi^*,\delta_\ell^j\rangle
=\alpha\rr^\top\delta_\ell^j<0.
\]
Under (ii), the current term \rev{is} nonpositive. Moreover,~\eqref{eq:hybrid-strict-source-loss} gives $\delta_{p_{s_\kappa}}^j<0$. Since $s_\kappa$ is controllable, its active-injection term is strictly negative.

Thus $\langle\nabla f(z^*),\delta^j\rangle<0$ for every branch. Together with~\eqref{eq:hybrid-tangent-condition}, Proposition~\ref{prop:isolating-direction} yields $\lambda_\non^*\in\R_{++}^n$. The conclusions follow from Theorem~\ref{thm:positive-multiplier-transfer}.
\end{proof}
\rev{The local sign tests $\mathcal F_\kappa$ and $\mathcal C_\kappa$ are therefore part of the certificate rather than sufficient conditions on their own: the interface adjustments, the tangent condition~\eqref{eq:hybrid-tangent-condition}, and the objective conditions are checked on the completed full-network directions. For the box~\eqref{eq:operating-set}, the tangent condition on a passive direction requires headroom below the upper voltage bound wherever $\delta_u^j>0$ and above the lower injection bound wherever $\delta_p^j<0$ or $\delta_q^j<0$. On a dispatchable direction it reduces to Remark~\ref{rem:fully-controllable-bounds}.}

\subsection{Placement of Controllable DERs}
\label{sec:der-placement}

\rev{Theorem~\ref{thm:hybrid-exactness} also indicates where controllable DERs are most useful.} The placement of controllable DERs can be used to manage reverse power flows. \rev{By step (i) of the partition rule, a branch whose two end buses both have controllable injections is dispatchable. Its test direction is the voltage-preserving direction~\eqref{eq:voltage-preserving-direction}, which involves no sign condition, so the reverse flow on that branch enters no test $\mathcal F_\kappa$ or $\mathcal C_\kappa$. The branch only needs the injection headroom of Remark~\ref{rem:fully-controllable-bounds}.} One possible design is \rev{therefore} to place controllable DERs along a backbone so that branches expected to carry reverse flows are dispatchable, while the attached passive subtrees operate without reverse flows.

A particularly simple structure is a dispatchable backbone connected to the slack bus, with terminal passive subtrees attached to it, as illustrated in Fig.~\ref{fig:hybrid-partitions}(b). For a branch in a passive subtree, the local direction changes no child-root voltages, so it can be completed by adjusting only the injection at the subtree root. The sign conditions $\mathcal F_\kappa$ or $\mathcal C_\kappa$ can therefore be checked separately on each passive subtree.

\begin{col}[Dispatchable backbone with terminal passive subtrees]
\label{col:terminal-passive-exactness}
Suppose the dispatchable branches form a backbone connected to the slack bus and every passive subtree is terminal. Under the assumptions of Theorem~\ref{thm:hybrid-exactness}, the conclusions of that theorem hold\rev{, and for a passive branch $j\in\mathcal N_\kappa$ the direction $\delta^j$ changes branch variables only within $\kappa$ and injections only at its root $s_\kappa$}.
\end{col}

\begin{proof}
For every passive subtree $\kappa$, terminality gives $\mathcal K^{\mathrm{des}}(\kappa)=\varnothing$. Thus, the construction in Lemma~\ref{lem:hybrid-global-extension} requires no descendant responses\rev{, and~\eqref{eq::adjustments} reduces to the root adjustment $\delta_{p_{s_\kappa}}^j=\delta_{p_0^\kappa}^j$}. The conclusion follows from Theorem~\ref{thm:hybrid-exactness}.
\end{proof}

\section{Conclusion}\label{sec:conclusion}

This paper developed a compositional exactness certificate for conic
relaxations of the DistFlow model in partially dispatchable radial networks.
The subtree voltage stability region of
Definition~\ref{def:strong-voltage-regularity} ensures power flow regularity
for the network and all compatible subtrees. By
Proposition~\ref{thm:wang-strong-regularity} and
Corollary~\ref{col:continuation-strong-region}, it contains the regions
certified by existing sufficient conditions for power flow
regularity~\cite{aolaritei2018hierarchical,wang2016necessary,cui2018new}.
Within this region, two endpoint sign conditions give complementary monotone
responses to root voltages. Corollary~\ref{prop:slack-monotonicity-injections}
shows that the limited-reverse-flow condition $\mathcal F$ signs every branch
current, while Corollary~\ref{col:gan-slack-monotonicity} shows that
$\mathcal C$, which the classical path condition implies by
Lemma~\ref{lemma:gan-C1-ancestor-sums}, signs voltages and weighted total
losses. A third response, with receiving-end powers fixed, requires regularity
only and serves dispatchable subtrees, as stated in
Corollary~\ref{prop:slack-monotonicity-receiving}.

Theorem~\ref{thm:positive-multiplier-transfer} and
Proposition~\ref{prop:isolating-direction} recast the classical exactness
arguments as sign tests on individual branch multipliers, which turns
exactness into a branchwise property that can be certified region by region.
Dispatchable subtrees use the voltage-preserving directions of
Lemma~\ref{prop:voltage-preserving-direction}, and passive subtrees use
fixed-injection directions, which Lemma~\ref{lem:hybrid-global-extension}
extends through their descendants and completes with controllable interface
injections. Theorem~\ref{thm:hybrid-exactness} shows that, together with
tangency to the operating set and objective monotonicity, these directions
yield strictly positive multipliers, global optimality, and exactness of every
relaxed optimizer, with uniqueness when the slack voltage is fixed. The
construction also informs DER placement. Under the partition rule of
Section~\ref{sec:hybrid-partition}, controllable injections at both ends of a
reverse-flow branch exempt it from the passive sign conditions, and
Corollary~\ref{col:terminal-passive-exactness} shows that a dispatchable
backbone with terminal passive subtrees keeps every test direction local.

The certificates combine a priori and a posteriori checks. The sign condition
$\mathcal C$ follows from the classical path condition evaluated on injection
and voltage bounds, including its sufficient forms with constant or monotone
resistance-to-reactance ratios (Lemma~\ref{lemma:gan-C1-ancestor-sums} and
Remark~\ref{rem::sufficient::C}). It can therefore be verified before the OPF
is solved. Regularity, tangency, and the objective conditions, by contrast, are
evaluated at a KKT point of the original problem, and failure of any test is
inconclusive. Several questions remain open. First, a priori conditions for
these remaining checks would turn the certificate from a check into a design
tool. Second, the tree path structure underlying both sign conditions does not
carry over to meshed grids, where the branch-flow relaxation must also account
for angle constraints around cycles. Third, conic relaxations can be inexact in
practice, for example in the calculation of operating
envelopes~\cite{moring2026exploring}. When a regional test fails, it remains
open whether the decomposition can localize the source of inexactness and bound
its effect. Finally, extensions to unbalanced networks, shunt elements, and
tap-changing devices remain to be developed.

\begin{appendices}
\section{Proofs}
\label{app:proofs}
\rev{This appendix collects the verifications whose only use in the main text is the statement proved: the comparison with the path products of~\cite{gan2014exact}, the phasor argument behind Proposition~\ref{thm:wang-strong-regularity}, two strictness arguments, and the sign bookkeeping of Sections~\ref{sec:passive-fixed-injection} and~\ref{sec:hybrid-extension}.}

\subsection{Proof of Lemma~\ref{lemma:gan-C1-ancestor-sums}}
\label{app:proof-gan-C1-ancestor-sums}
\begin{proof}
Use the C1 bounds $p\le\overline p$, $q\le\overline q$, and $u\ge\underline u$, with $\underline u\in\R^n_{++}$. Define the two-vectors
\[
 \xi_j:=\begin{bmatrix}\rr_j\\\xx_j\end{bmatrix},\quad
 w_j:=\begin{bmatrix}P_j\\Q_j\end{bmatrix},\quad
 \overline w_j:=\begin{bmatrix}
 ([-\hat C^{-1}\overline p]_j)^+\\
 ([-\hat C^{-1}\overline q]_j)^+
 \end{bmatrix},
\]
where $a^+:=\max\{a,0\}$. The bound-based and operating-point path matrices are
\begin{align*}
 \underline A_j&:=\mathbf I_2-2\xi_j\overline w_j^\top/\underline u_j,\\
 A_j&:=\mathbf I_2+2\xi_jw_j^\top/u_j.
\end{align*}
For $i\in\mathcal N_j^{\anc+}$, let
$\mathcal N^{\seg}_{ij}:=\mathcal N_j^{\anc}\setminus\mathcal N_i^{\anc}$.
Products are ordered upstream to downstream, with the empty product equal to $\mathbf I_2$. C1~\cite[Thm.~1]{gan2014exact} is equivalently
\begin{equation*}
 \left(\prod_{k\in\mathcal N^{\seg}_{ij}}\underline A_k\right)\xi_j\in\R^2_{++},
 \qquad i\in\mathcal N_j^{\anc+},
\end{equation*}
because every ancestor segment lies on a root-to-leaf path.

The affine balances give $P=\hat C^{-1}p+\Gamma\mathbf R\ell$ and the analogous identity for $Q$. Since $\hat C^{-1}\in\R^{n\times n}_{-}$, $\Gamma\in\R^{n\times n}_{+}$, and $\ell\in\R^n_{+}$, they imply $w_j\ge-\overline w_j$. Consequently,
\begin{align*}
 \zeta_j&:=2(w_j/u_j+\overline w_j/\underline u_j)\in\R^2_{+},\\
 A_j&=\underline A_j+\xi_j\zeta_j^\top.
\end{align*}
The path-product comparison in~\cite[App.~B, Lem.~3]{gan2014exact} now yields
\begin{equation*}
 \left(\prod_{k\in\mathcal N^{\seg}_{ij}}A_k\right)\xi_j\in\R^2_{++}.
\end{equation*}
\rev{Indeed,} expand the product using $A_k=\underline A_k+\xi_k\zeta_k^\top$. The term with no rank-one factor is strictly positive by C1. Every other term is a positive C1 segment vector multiplied by factors of the form $\zeta_a^\top(\prod\underline A_k)\xi_b\ge0$, where $b$ is the next selected downstream index, or $j$ after the last selected index. The product follows the segment strictly between $a$ and $b$. Thus all remaining terms are nonnegative.

It remains to identify these products with the entries in~\eqref{eq:T}. Fix $j$ and put $b=\hat{\mathbf U}^{-1}\mathbf e_j$. By~\eqref{eq::hatU}, its rows satisfy
\begin{equation}\label{eq:gan-inverse-column-recursion}
 u_i b_i=\delta_{ij}
 +2w_i^\top\sum_{k:\,i\in\mathcal N_k^{\anc}}b_k\xi_k.
\end{equation}
Backward substitution gives
\begin{equation}\label{eq:gan-inverse-column-support}
 b_i=0\quad(i\notin\mathcal N_j^{\anc+}),\qquad b_j=1/u_j\rev{,}
\end{equation}
\rev{since} a bus outside the ancestor path of $j$ has no descendant on that path, so its row has zero right-hand side after its descendants have been eliminated.

For the path $i=k_0\to\cdots\to k_m=j$, define
$v_t:=\sum_{s=t}^{m}\xi_{k_s}b_{k_s}$. Equation~\eqref{eq:gan-inverse-column-recursion} gives
\begin{align*}
 v_m&=\xi_j/u_j,\\
 v_t&=v_{t+1}+2\xi_{k_t}w_{k_t}^\top v_{t+1}/u_{k_t}
     =A_{k_t}v_{t+1}.
\end{align*}
Combining this recursion with the path entries of $-\hat C^{-1}$ yields
\begin{equation*}
 \begin{aligned}
 \begin{bmatrix}[\mathbf T_r]_{ij}\\{}[\mathbf T_x]_{ij}\end{bmatrix}
 &=\sum_{k\in\mathcal N_j^{\anc+}\setminus\mathcal N_i^{\anc}}
       \xi_k[\hat{\mathbf U}^{-1}]_{kj}\\
 &=\frac1{u_j}\left(\prod_{k\in\mathcal N^{\seg}_{ij}}A_k\right)\xi_j\in\R^2_{++}.
 \end{aligned}
\end{equation*}
If $i\notin\mathcal N_j^{\anc+}$, the support condition~\eqref{eq:gan-inverse-column-support} instead makes both entries zero. Therefore $\mathbf T_r,\mathbf T_x\in\R^{n\times n}_{+}$, proving $z\in\mathcal C$.
\end{proof}
\subsection{Proof of Proposition~\ref{thm:wang-strong-regularity}}
\label{app:wang-strong-regularity}
\begin{revblock}
\begin{proof}[Proof sketch]
\emph{Step 1: phasor coordinates.} Fix $V_0:=\sqrt{u_s}$ and write $S_j:=P_j+\mathbf jQ_j$. At a feasible point with $u\in\R^n_{++}$, the recursion
\[
 V_j:=V_{\pi(j)}/a_j,\qquad a_j:=1+\zz_j\overline{S_j}/u_j,
\]
is well defined, because~\eqref{eq::distflow} gives $|a_j|^2=u_{\pi(j)}/u_j>0$. It yields $|V_j|^2=u_j$, branch currents $I_j:=(V_{\pi(j)}-V_j)/\zz_j$ with $V_j\overline{I_j}=S_j$ and $|I_j|^2=\ell_j$, and nodal injection currents $[I_{\mathrm b}]_j:=\overline{s_j/V_j}$ with $|[I_{\mathrm b}]_j|=\iota_j$ and $V=V_0\mathbf1+\Zb I_{\mathrm b}$. Conversely, phasors determine $(P,Q,u,\ell)$ smoothly. The branch equations~\eqref{eq:DistFlow-matrix:U} and~\eqref{eq:DistFlow-matrix:nonlinear} therefore form a $2n$-dimensional manifold parametrized by $V\in\C^n$: their derivative has rank $2n$, since the columns $(P_j,Q_j,\ell_j)$ occur only in the two rows of branch $j$, with entries $(2\rr_j,2\xx_j,|\zz_j|^2)$ and $(2P_j,2Q_j,-u_j)$, and these are proportional only if $S_j=-u_j\zz_j/|\zz_j|^2$, which forces $u_{\pi(j)}=0$. The balances~\eqref{eq:DistFlow-matrix:P}--\eqref{eq:DistFlow-matrix:Q} coincide with the nodal injection map $s(V)=V\circ\overline{I_{\mathrm b}(V)}$. Hence, with $(p,q)$ fixed, $J(z;u_s)$ is nonsingular if and only if the real $2n\times2n$ Jacobian of $V\mapsto s(V)$ at fixed $V_0$ is nonsingular. The same holds on every connected subtree with its own root phasor.

\emph{Step 2: pointwise nonsingularity.} Under~\eqref{eq:wang-condition}, a variation with $\delta s=0$ and $\delta V_0=0$ satisfies $\delta[I_{\mathrm b}]_i=-\bigl([I_{\mathrm b}]_i/\overline{V_i}\bigr)\overline{\delta V_i}$ and $\delta V=\Zb\,\delta I_{\mathrm b}$. If $\eta:=\max_i|\delta V_i|/|V_i|>0$, an attaining index $i$ gives
\[
 \eta|V_i|\le\sum_{j\in\mathcal N}|[\Zb]_{ij}|\,|[I_{\mathrm b}]_j|\,\frac{|\delta V_j|}{|V_j|}<\eta|V_i|,
\]
a contradiction. Thus $\delta V=0$, and $J(z;u_s)$ is nonsingular by Step~1. This is the diagonal-dominance argument of~\cite[Thm.~2]{wang2016necessary}.

\emph{Step 3: a common path.} Put $\Delta:=\Zb I_{\mathrm b}$ and $R_i:=\sum_j|[\Zb]_{ij}|\,|[I_{\mathrm b}]_j|$, so that $|\Delta_i|\le R_i<|V_i|$. Scale the injection currents,
\begin{equation}\label{eq:app-current-path}
 I_{\mathrm b}(t):=tI_{\mathrm b},\qquad V(t):=V_0\mathbf1+t\Delta,\qquad t\in[0,1],
\end{equation}
as in~\cite[Thm.~2]{cui2018new}. Since $V_i(t)=V_i-(1-t)\Delta_i$,
\[
 |V_i(t)|\ge|V_i|-(1-t)R_i>tR_i=\sum_{j\in\mathcal N}|[\Zb]_{ij}|\,|[I_{\mathrm b}(t)]_j|,
\]
so voltages stay positive and~\eqref{eq:wang-condition} holds along the path, with injections $s(t)=V(t)\circ\overline{I_{\mathrm b}(t)}$ that vary with $t$. Mapping back through Step~1 gives an exact DistFlow path $z(t)$ at fixed $u_s$ from $z^0$ ($t=0$) to $z$ ($t=1$), on which $J$ is nonsingular by Step~2. Hence $z\in\mathcal R_{\mathrm{vs}}$.

\emph{Step 4: compatible subtrees.} Fix $\kappa\in\mathcal S(\mathcal G)$. For $a\in\mathcal N_\kappa$, let $D_\kappa(a)$ consist of $a$ and all buses reached from $a$ through omitted outgoing branches. These sets are disjoint. By Kirchhoff's law, the current on an omitted branch $k$ with $\pi(k)=a$ equals minus the sum of the injection currents below it, so subtracting the omitted sending-end powers in~\eqref{eq:compatible-subtree-injections} gives the compatible injection currents
\[
 [I_{\mathrm b}^\kappa(t)]_a=t\sum_{j\in D_\kappa(a)}[I_{\mathrm b}]_j,\qquad a\in\mathcal N_\kappa,
\]
with $V_\kappa(t)=V_{s_\kappa}(t)\mathbf1+\Zb^\kappa I_{\mathrm b}^\kappa(t)$, where $\Zb^\kappa$ is~\eqref{eq:bus-impedance} formed on $\kappa$ with root $s_\kappa$. Branches omitted at $s_\kappa$ only change the local slack balance. For $i,a\in\mathcal N_\kappa$ and $j\in D_\kappa(a)$, the common path of $i$ and $a$ from $s_\kappa$ is a subset of the common root path of $i$ and $j$; since $\rr,\xx\ge0$, $|[\Zb^\kappa]_{ia}|\le|[\Zb]_{ij}|$. By disjointness,
\[
 \sum_{a\in\mathcal N_\kappa}|[\Zb^\kappa]_{ia}|\,|[I_{\mathrm b}^\kappa(t)]_a|
 \le t\sum_{a\in\mathcal N_\kappa}\sum_{j\in D_\kappa(a)}|[\Zb]_{ij}|\,|[I_{\mathrm b}]_j|
 \le tR_i<|V_i(t)|.
\]
Thus~\eqref{eq:wang-condition} holds on $\kappa$ with root phasor $V_{s_\kappa}(t)\ne0$, and Steps~1--2 applied to $\kappa$ give $\det J_\kappa(z_\kappa(t);u_{s_\kappa}(t))\ne0$ for all $t\in[0,1]$. Since $\kappa$ was arbitrary, the path~\eqref{eq:app-current-path} lies in the set of Definition~\ref{def:strong-voltage-regularity}, and $z\in\mathcal R_{\mathrm{svs}}$.
\end{proof}
\end{revblock}

\subsection{Proof of~\eqref{eq:gan-strict-weighted-inverse} in Proposition~\ref{prop:gan-transformed-M}}\label{app:proof-prop-gan-transformed-M}
\begin{proof}
\rev{With $\mathbf M_\ell^{-\top}\xi\in\R^n_+$ already established in the main text, it remains to show strict positivity.} For \rev{the strict statement, fix} $\xi\in\{\mathscr r,\mathscr x\}$ \rev{and} set
\[
\begin{aligned}
\mathbf A&:=\hat C^{-\top}\hat{\mathbf Z}\hat{\mathbf U}^{-1}\in\R^{n\times n}_{+},\\
b&:=\hat{\mathbf U}^{-\top}\xi\in\R^n_{+},\qquad v:=\mathbf M_\ell^{-\top}\xi\in\R^n_{+}.
\end{aligned}
\]
Since $\mathbf W=\mathbf I-\mathbf L\mathbf A$, we have $v=b+\mathbf A^\top\mathbf L v$. We show that $v_j>0$ whenever $\xi_j>0$. Suppose otherwise, and choose the smallest index $j$ with $\xi_j>0$ and $v_j=0$ in the topological ordering. If $\xi_i\ell_i=0$ for every strict ancestor $i$ of $j$, then each contribution $\xi_i[\hat{\mathbf U}^{-1}]_{ij}$ vanishes: either $\xi_i=0$, or $\ell_i=0$ forces $P_i=Q_i=0$ by~\eqref{eq:DistFlow-nonlinear}, making row $i$ of $\hat{\mathbf U}^{-1}$ diagonal. Thus $b_j=\xi_j/u_j>0$, a contradiction. Otherwise, let $k$ be the deepest strict ancestor of $j$ with $\xi_k\ell_k>0$, and let $h$ be its child on the path to $j$. The same argument on the intervening path gives $[\mathbf T_\xi]_{hj}=\xi_j/u_j$. The factorization in the proof of Lemma~\ref{lem:gan-triangular-signs} yields
\[
\begin{aligned}
\mathbf A&=(-\hat C^{-\top})\bigl[\mathbf R(\mathbf I_n+\hat C_f)\mathbf T_r+\mathbf X(\mathbf I_n+\hat C_f)\mathbf T_x\bigr],\\
A_{kj}&\ge\xi_k[\mathbf T_\xi]_{hj}=\frac{\xi_k\xi_j}{u_j}>0.
\end{aligned}
\]
Since $k<j$ and $\xi_k>0$, we have $v_k>0$, and hence $v_j\ge A_{kj}\ell_kv_k>0$, again a contradiction. Therefore $[\mathbf M_\ell^{-\top}\xi]_j>0$ whenever $\xi_j>0$. For $\alpha,\beta\ge0$, each positive component of $\alpha\mathscr r+\beta\mathscr x$ consequently gives a positive component of $\alpha\mathbf M_\ell^{-\top}\mathscr r+\beta\mathbf M_\ell^{-\top}\mathscr x$, proving~\eqref{eq:gan-strict-weighted-inverse}.
\end{proof}

\subsection{Proof of~\eqref{eq::strict::ineq} in Corollary~\ref{col:gan-slack-monotonicity}}\label{app:proof-col-gan-slack-monotonicity}
\begin{proof}
\rev{With $b:=\hat{\mathbf U}^{-\top}\zeta$ and the derivative formula from the main text, it remains to characterize strictness.} Since $\mathbf K_u\ell,\,b,\,\ell\in\R^n_{+}$, this derivative is strictly negative if and only if $b^\top\ell>0$. It remains to show that $b^\top\ell=0$ if and only if $\zeta^\top\ell=0$. Suppose $b^\top\ell=0$. Then $b_j=0$ whenever $\ell_j>0$, while $\ell_j=0$ implies $P_j=Q_j=0$ by~\eqref{eq:DistFlow-nonlinear}. Thus $\mathbf P b=\mathbf Q b=0$, and~\eqref{eq:hatU} gives
\[
 \zeta=\hat{\mathbf U}^\top b
 =\bigl(\mathbf U-2\mathbf R\Gamma^\top\mathbf P
              -2\mathbf X\Gamma^\top\mathbf Q\bigr)b
 =\mathbf U b.
\]
Hence $\zeta^\top\ell=\sum_{j\in\mathcal N}u_j b_j\ell_j=0$. Conversely, if $\zeta^\top\ell=0$, the same argument gives $\mathbf P\mathbf U^{-1}\zeta=\mathbf Q\mathbf U^{-1}\zeta=0$, so
\[
 \hat{\mathbf U}^\top\mathbf U^{-1}\zeta=\zeta,
 \qquad
 b=\hat{\mathbf U}^{-\top}\zeta=\mathbf U^{-1}\zeta.
\]
Therefore $b^\top\ell=\sum_{j\in\mathcal N}\zeta_j\ell_j/u_j=0$. This proves the equivalence and completes both strictness statements.
\end{proof}

\begin{revblock}
\subsection{Proof of Corollary~\ref{col:passive-multiplier-bound}}\label{app:proof-passive-multiplier-bound}
\begin{revblock}
\begin{proof}
By~\eqref{eq:affine-elimination-slack}, $\delta^j_x=\partial_\ell x^\star\,\delta^j_\ell$, so $\delta^j$ is tangent to the affine solution map at fixed $(p,q,u_s)$ and $\mathrm Dg_\lin\,\delta^j=0$. Moreover $\mathrm Dg_\non\,\delta^j=\partial_\ell\hat g_\non\,\delta^j_\ell=-\mathbf M_\ell^*\delta_\ell^j=\mathbf e_j$. Since $f$ does not depend on $(P,Q)$ directly and $\delta_p^j=\delta_q^j=0$,
\[
 -\langle\nabla f(z^*),\delta^j\rangle=-(c_\ell^*)^\top\delta_\ell^j=[(\mathbf M_\ell^*)^{-\top}c_\ell^*]_j.
\]
Proposition~\ref{prop:isolating-direction} with $\eta=1$ and the polarity~\eqref{eq:normal-tangent-polarity} give the lower bound, and Theorem~\ref{thm:passive-exactness} makes its right-hand side positive under either condition. The signs of $\delta_u^j=\mathbf K_u^*\mathbf e_j$ and $\zeta^\top\delta_\ell^j=-\mathbf e_j^\top(\mathbf M_\ell^*)^{-\top}\zeta$, $\zeta\in\{\rr,\xx\}$, follow from Proposition~\ref{thm:strong-region-M} under $\mathcal F$ and from Proposition~\ref{prop:gan-transformed-M} under $\mathcal C$. The slack-injection changes follow from the power balance~\eqref{eq::slack::injections} with $(p,q)$ fixed.
\end{proof}
\end{revblock}

\subsection{Proofs of Lemmas~\ref{lem:hybrid-descendant-response} and~\ref{lem:hybrid-global-extension}}\label{app:proof-hybrid-lemmas}
\begin{proof}[Proof of Lemma~\ref{lem:hybrid-descendant-response}]
\rev{Take $\delta_\sigma$ as $\tau_\sigma$ times the slack-voltage sensitivity of the corresponding local solution map, so that the local residual derivatives, including the root-voltage contribution, vanish.} The signs for $\delta_{u_{\sigma}},\delta_{\ell_\sigma}\in\R^{n_{\sigma}}$ in \rev{(i), (ii), and (iii)} are precisely the conclusions of Corollaries~\ref{prop:slack-monotonicity-injections}, \ref{col:gan-slack-monotonicity} and~\ref{prop:slack-monotonicity-receiving}.
In \rev{(iii)}, the signs for $\delta_{p^{\sigma}},\delta_{\rev{q}^\sigma}\in\R^{n_{\sigma}}$ come from the sensitivity~\eqref{eq::dispatch::pq} and the nonnegative inverse of $\mathbf{M}_{\ell,\sigma}^r$ by Proposition~\ref{prop::nonsingular::Mr}.

For the compatible local slack injections $\delta_{p_{0}^{\sigma}},\delta_{q_{0}^{\sigma}}$, \rev{nonpositivity} in (i) and (ii) comes from the \rev{subtree} power balance~\eqref{eq::slack::injections::local} with $(p^\sigma,q^\sigma)$ fixed, i.e.,
\[
\delta_{p_{0}^{\sigma}}= \rr_\sigma^\top\delta_{\ell_\sigma}\le0,\qquad \delta_{q_{0}^{\sigma}}= \xx_\sigma^\top\delta_{\ell_\sigma}\le0,
\]
while \rev{nonpositivity} in (iii) comes from the nodal balance at the local \rev{slack}~\eqref{eq::slack::balance::local} with $(P_\sigma,Q_\sigma)$ fixed, i.e.,
\[
\delta_{p_{0}^{\sigma}}=\hspace{-4pt}\sum_{\substack{k:\pi(k)=s_\sigma\\ \rev{k\in\mathcal N_\sigma}}}\hspace{-2pt}\rr_k[\delta_{\ell_\sigma}]_k\le0,\;
\delta_{q_{0}^{\sigma}}=\hspace{-4pt}\sum_{\substack{k:\pi(k)=s_\sigma\\ \rev{k\in\mathcal N_\sigma}}}\hspace{-2pt}\xx_k[\delta_{\ell_\sigma}]_k\le0,
\]
which \rev{completes} the proof.
\end{proof}

\begin{proof}[Proof of Lemma~\ref{lem:hybrid-global-extension}]
Fix $j\in\mathcal N_\kappa$. Apply~\eqref{eq::direction::passive} to the source subtree $\kappa$, with $\delta_{u_{s_\kappa}}^j=0,\;
\delta_{p^\kappa}^j=\delta_{q^\kappa}^j=0,$ and
\[
\delta_{\ell_\kappa}^j
=-\mathbf M_{\ell,\kappa}^{-1}\mathbf e_{j,\kappa},
\qquad
\delta_{u_\kappa}^j
=\mathbf K_{u,\kappa}\mathbf e_{j,\kappa}\rev{.}
\]
This direction has zero local affine residual derivative and nonlinear residual derivative $\mathbf e_{j,\kappa}$, as \rev{shown} in the proof of Corollary~\ref{col:passive-multiplier-bound}. By Lemma~\ref{lem:strong-subtree-inheritance} and Propositions~\ref{thm:strong-region-M} and~\ref{prop:gan-transformed-M},
\[
\delta_{u_\kappa}^j\in\R^{n_\kappa}_+,
\qquad
\rr_\kappa^\top\delta_{\ell_\kappa}^j\le0,
\qquad
\xx_\kappa^\top\delta_{\ell_\kappa}^j\le0.
\]
Since the compatible injections are fixed, differentiating~\eqref{eq::slack::injections::local} gives
\[
\delta_{p_0^\kappa}^j
=\rr_\kappa^\top\delta_{\ell_\kappa}^j,
\qquad
\delta_{q_0^\kappa}^j
=\xx_\kappa^\top\delta_{\ell_\kappa}^j.
\]
Proceed through the descendants in parent-to-child order. For each descendant $\sigma$, set $\tau_\sigma^j$ equal to the already assigned voltage change at $s_\sigma$ of its parent subtree, and apply Lemma~\ref{lem:hybrid-descendant-response}. Its voltage conclusion ensures that $\tau_\sigma^j\ge0$ at every stage. Each response preserves the local DistFlow equations~\eqref{eq:DistFlow-matrix} to first order, including the root-voltage contribution, and has nonpositive compatible-injection, slack-input, and weighted-loss changes. On every subtree outside $\{\kappa\}\cup\mathcal K^{\mathrm{des}}(\kappa)$, set the variations of $(P,Q,u,\ell)$ and the compatible injections to zero.

Assemble the branch variables using the partition and define the physical injection changes by~\eqref{eq::adjustments}. Every term is nonpositive. Only controllable interface buses or buses in dispatchable subtrees require injection adjustments, so fixed nodal injections remain unchanged.

Matching the shared root voltages preserves the voltage equations~\eqref{eq:DistFlow-matrix:U}, and the injection adjustments~\eqref{eq::adjustments} restore all nodal balances~\eqref{eq:DistFlow-matrix:P}\eqref{eq:DistFlow-matrix:Q}. The source contributes local nonlinear residual~\eqref{eq:DistFlow-matrix:nonlinear} by $\mathbf e_{j,\kappa}$, while all other subtrees contribute zero. Thus~\eqref{eq:hybrid-global-isolation} holds. The voltage signs follow from the construction, and summing the nonpositive subtree loss changes gives
\[
\rr^\top\delta_\ell^j\le0,
\qquad
\xx^\top\delta_\ell^j\le0.
\]
Finally, suppose $\rr_\kappa\in\R_{++}^{n_\kappa}$. Under $\mathcal F_\kappa$, the inverse $\mathbf M_{\ell,\kappa}^{-1}$ is nonnegative with positive diagonal. Under $\mathcal C_\kappa$, Proposition~\ref{prop:gan-transformed-M} gives~\eqref{eq:gan-strict-weighted-inverse}. In either case,
\[
\mathbf M_{\ell,\kappa}^{-\top}\rr_\kappa
\in\R_{++}^{n_\kappa},
\]
and hence
\[
\rr_\kappa^\top\delta_{\ell_\kappa}^j
=-\mathbf e_{j,\kappa}^\top
\mathbf M_{\ell,\kappa}^{-\top}\rr_\kappa<0.
\]
At $s_\kappa$, differentiating the local power balances of the subtree~\eqref{eq::slack::injections::local} gives
\[
\delta_{p_{s_\kappa}}^j
=\delta_{p_0^\kappa}^j
=\rr_\kappa^\top\delta_{\ell_\kappa}^j<0\rev{.}
\]
For all other \rev{subtrees}, the changes in active power loss are nonpositive, so $\rr^\top\delta_\ell^j<0$.
\end{proof}

\end{revblock}
\end{appendices}

\bibliography{root}
\end{document}